\documentclass{article}
 \usepackage{xcolor}
 \definecolor{blue(ryb)}{rgb}{0.01, 0.28, 1.0}
\definecolor{brandeisblue}{rgb}{0.0, 0.44, 1.0}
\definecolor{ceruleanblue}{rgb}{0.16, 0.32, 0.75}
\definecolor{cobalt}{rgb}{0.0, 0.28, 0.67}
\definecolor{coolblack}{rgb}{0.0, 0.18, 0.39}
\definecolor{darkblue}{rgb}{0.0, 0.0, 0.55}
\usepackage[hidelinks]{hyperref}
\hypersetup{
    colorlinks,
    citecolor=darkblue,
    filecolor=black,
    linkcolor=darkblue,
    urlcolor=black
}
\usepackage{amsmath, amssymb, amsfonts, amsthm}
\usepackage{graphicx}
\def\nz{\ifmmode {I\hskip -3pt N} \else {\hbox {$I\hskip -3pt N$}}\fi}
\def\zz{\ifmmode {Z\hskip -4.8pt Z} \else

       {\hbox {$Z\hskip -4.8pt Z$}}\fi}
\def\qz{\ifmmode {Q\hskip -5.0pt\vrule height6.0pt depth 0pt
       \hskip 6pt} \else {\hbox
       {$Q\hskip -5.0pt\vrule height6.0pt depth 0pt\hskip 6pt$}}\fi}
\def\rz{\ifmmode {I\hskip -3pt R} \else {\hbox {$I\hskip -3pt R$}}\fi } 
\def\cz{\ifmmode {C\hskip -4.8pt\vrule height5.8pt \hskip 6.3pt} \else 
{\hbox {$C\hskip -4.8pt\vrule height5.8pt \hskip 6.3pt$}}\fi} 

\def\curl{{\rm curl}\,}

\def \Ab{{\bf A}}

\def\Og {{\cal O}} 
\def\Nb {{\bf N}} 

\def\and {{\rm \; and \;}}
\def\dist {{\rm \; dist \;}}

\def\Bb{\mathbf B}
\def\Tb{\mathbf T}
\def\Vb{\mathbf V}
\def\Nb{\mathbf N}
\def\nb{\mathbf n}
\def\R{\mathbb R}
\def\Z{\mathbb Z}
\def\C{\mathbb C}
\definecolor{ao(english)}{rgb}{0.0, 0.5, 0.0}

\newcommand{\beq}{\begin{equation}}
\newcommand{\eeq}{\end{equation}}
\newcommand {\pa}{\partial}

\renewcommand*{\overrightarrow}[1]{\vbox{\halign{##\cr 
  \tiny\rightarrowfill\cr\noalign{\nointerlineskip\vskip1pt} 
  $#1\mskip2mu$\cr}}}

\makeatletter
\@addtoreset{equation}{section}

\makeatother

\newtheorem{theorem}{Theorem}[section]

\newtheorem{lemma}[theorem]{Lemma}
\newtheorem{proposition}[theorem]{Proposition}

\newtheorem{remark}[theorem]{Remark}

\newtheorem{assumption}[theorem]{Assumption}

\title{\bf  Helical magnetic fields and  semi-classical asymptotics of the lowest 
 eigenvalue }

\author{B. Helffer, A. Kachmar}

\begin{document}
\bibliographystyle{plain}
\maketitle 

%
\begin{abstract}
We study the 3D Neuman magnetic Laplacian in the presence  of a semi-classical parameter and a non-uniform magnetic field with constant intensity. We determine a sharp two term asymptotics for the lowest eigenvalue, where the second term involves a quantity related to the magnetic field and the  geometry of the domain. In the special case of the unit ball  and a helical magnetic field,  the concentration takes place  on two symmetric points of the  unit sphere. 
\end{abstract}
\section{Main results}
Let   $\Omega\subset\R^3$ be  an open  and bounded set   with  a smooth boundary  $\partial\Omega$.  Let  us consider a smooth  magnetic field $\Bb:\overline{ \Omega} \to\R^3$  (so $\Bb$ should be closed) which will always be assumed to satisfy
\begin{equation}\label{eq:B}
\forall\,x\in\Omega, \quad |\Bb(x)|=b
\end{equation}
where $b>0$ is a constant.   Without loss of generality, we assume from now on that $b=1$.
  Let $\Ab(x)$ be a magnetic potential such that
\begin{equation}\label{eq:A}
\curl \Ab = \Bb\;.
\end{equation}
We are interested in the analysis of the lowest eigenvalue
 $\lambda_1( \Ab,h)$ of 
 the Neumann realization of the Schr\"odinger operator in $\Omega $ with magnetic field
\begin{equation}\label{eq:op}
P_\Ab^h:=\Delta_{h,\Ab} = \sum_{j=1}^3 (hD_{x_j} + A_j(x))^2\;.
\end{equation} 
  
We introduce the following assumptions. 
\begin{assumption}[C1]\label{ass:C1}~\\
 The set of boundary points where $\Bb$ is
tangent to $ \partial \Omega$, i.e.
\begin{align}\label{eq:1.4}
\Gamma:= \{ x \in \partial \Omega\,\big|\, \Bb \cdot \Nb(x) = 0 \},
\end{align}
is a regular submanifold of $\partial \Omega$~:
\begin{equation}\label{eq:k-n,B}
\kappa_{n,\Bb}(x):= |d ^T (\Bb \cdot \Nb) (x)|  \neq 0\;,\; \forall x\in \Gamma\;.
\end{equation}
Here $d^T$ is the differential defined on functions on $\pa \Omega$ and  $\mathbf N(x)$ is the unit inward normal of $\Omega$.
\end{assumption}

\begin{assumption}[C2]\label{ass:C2}~\\
 The set of points where $ \Bb$ is
tangent to $ \Gamma$ is finite.
\end{assumption}
 
These assumptions are rather generic and  
for instance satisfied for ellipsoids, when $\Bb$ is constant. When  $|\Bb|$ is constant, the above assumptions hold for the sphere with a helical magnetic field (see Sec.~\ref{sec:helix}).

Let us introduce the  constant  $  \widehat{\gamma}_{0,\Bb}$ involving the ``magnetic curvature'' in \eqref{eq:k-n,B},  which 
  is defined by
\begin{equation}\label{eq:hat-gam}
\widehat{\gamma}_{0,\Bb}:= \inf_{x\in \Gamma}
\widetilde{\gamma}_{0,\Bb}(x),\end{equation}
where
\begin{equation}\label{eq:tilde-gam}
\widetilde{\gamma}_{0,\Bb}(x) := 2^{-2/3} \widehat{\nu}_0 \delta_0^{1/3} 
|\kappa_{n,\Bb}(x)|^{2/3}
\Big(1 - (1-\delta_0)| \Tb(x)\cdot \Bb(x) |^2
\Big)^{1/3}\;.
\end{equation}
Here  $\Tb(x)$ is the oriented,
unit tangent vector to $ \Gamma$ at the point $
x$,
$\delta_0\in ]0,1[$ and $ \widehat \nu_0>0$ are spectral quantities  relative to the De Gennes and Montgomery operators which will be   introduced in  \eqref{eq:Th0} and \eqref{eq:Mont}.

When $\Bb$ is constant,  the following two-term asymptotics of  $ \lambda_1(B)$ has been established
by Helffer-Morame \cite{HelMo4} and Pan \cite{Pan3D}.
\begin{theorem}\label{thabove}~\\
Let us assume that $\Bb$ is constant.
Then,  if $
\Omega$  and $ \Bb$ satisfy {\rm (C1)-(C2)}, there exists $\eta >0$ such that the lowest eigenvalue $\lambda_1^N(\Ab,h)$  satisfies as
$ h \rightarrow  0$ 
\begin{equation}\label{2Texp} 
\lambda_1^N(\Ab,h)= \Theta_0 h + \widehat{\gamma}_{0,\Bb}  h^{\frac{4}{3}} + {\mathcal O}(h^{\frac{4}{3} + \eta})\,.
\end{equation}
\end{theorem}
The aim of  this paper is to  prove that Theorem \ref{thabove} also holds under the weaker assumption that $|\Bb|$ is constant. 

\begin{theorem}\label{thm:main}~\\
Under the assumptions {\rm (C1)-(C2)}, if $|\Bb|$ is constant,  then the asymptotics in \eqref{2Texp} holds for the lowest eigenvalue  $\lambda_1^N(\Ab,h)$.
\end{theorem}

An interesting example of a non-constant magnetic field but with a constant intensity is the helical magnetic field occurring in the theory  of liquid crystals.   Up to the action of an orthogonal matrix,  it can be expressed as follows \cite{Pan8}
\begin{equation}\label{eq:B=n-tau}
\Bb=\curl\nb_\tau=-\tau\nb_\tau,\quad  \nb_\tau= \Big(\frac1\tau\cos (\tau x_3), \frac1\tau\sin (\tau x_3),0\Big).
\end{equation}
Here  $\tau>0$  is a given constant.   In this situation ($\Bb=-\tau\nb_\tau$), \cite{Pan8} derived  an upper bound on the eigenvalue $\lambda_1^N(\Ab,h)$, which is consistent with Theorem~\ref{thm:main}. Our  contribution is valid for a more general class of magnetic fields with constant intensity and also determines the asymptotically matching lower bound of the lowest eigenvalue.

\subsection*{Discussion and applications}
 The inspection of the eigenvalue $\lambda_1^N(\Ab,h)$ is vital in understanding the  transition between \emph{superconducting} and \emph{normal} states in the  Ginzburg-Landau model \cite{FoHe2}. In this context, the magnetic field is typically constant. Accurate estimates of the lowest eigenvalue $\lambda_1^N(\Ab,h)$ under constant magnetic fields  \cite{HelMo3,  HelMo4}  led to a precise understanding of the transition between superconducting and normal states \cite{FoHe1, FP}.

Non-homogeneous magnetic fields with constant intensity are encountered in the Landau--de\,Gennes theory  of liquid crystals, which is the analog of the  Ginzburg-Landau  theory  of  superconductivity. Here a  transition between \emph{smectic} and  \emph{nematic} phases occurs. Our main  result, Theorem~\ref{thm:main}, yields an accurate estimate of  the lowest eigenvalue $\lambda_1^N(\Ab,h)$ for magnetic fields with constant intensity, and by  analogy with \cite{FoHe1}, we expect it to yield a precise description of  the transition between surface smectic and nematic states (see \cite{Panlc}).

At  the threshold of the phase transition,  both superconductive and smectic states  nucleate  on the surface of the domain (near the curve $\Gamma$ introduced in \eqref{eq:Gam}). The paper \cite{Pan7}  contains  a nice discussion of this  interesting  analogy. The analysis of 3D surface superconductivity is the subject of  the  papers \cite{Pan3D,  FKP1, FMP}, while  surface smectics are rigorously studied in \cite{HePa2, FKP}. It  would be interesting to complete this analysis by  providing more accurate estimates at  the  threshold, where the linear analysis (such as the one in this paper) becomes handy.

The analysis in  this paper concerns the lowest eigenvalue. In the presence of a constant magnetic field,  and a ``single well'' assumption (i.e. the minimum in \eqref{eq:hat-gam} is non-degenerate and attained at a unique point),   accurate estimates of the low-lying eigenvalues 
were obtained recently in \cite{HR}. In our setting of a non-homogeneous magnetic field, the example of the ball under the helical magnetic field suggests the presence of multiple wells (see Remark~\ref{rem:tunneling}). 

The interaction between magnetic fields and 3D domains is interesting in other situations. In particular, for the Robin problem, we observe pure magnetic wells on the  surface of the domain  \cite{HKR}, and in  the case of a constant magnetic field, strong diamgnetism does not hold for the ball \cite{M}.

\subsection*{Organization and outline of the proof}

The proof of  Theorem~\ref{thm:main} is split into two parts. In the first part, we  establish  a lower bound
of the lowest eigenvalue, by comparing the quadratic form  via a simpler form related to a new  model operator. Comparing with  the constant magnetic field in \cite{HelMo4}, we prove that the model operator in  our setting is a  perturbation  of  the one considered  in \cite{HelMo4}.

The second part of the proof  is devoted to an upper bound of the lowest eigenvalue,   already studied  for $\Bb$ in \eqref{eq:B=n-tau}   \cite{Pan8}, but we revisit it since our formulation is not the same as  \cite{Pan8}.  The upper  bound follows after computing the quadratic form  of a suitable  trial  state, having the same structure as the constant magnetic field case in \cite{HelMo4, Pan3D}. However, there are additional terms in the computations due to the varying magnetic field,  which require a careful handling. 

The model operator takes into consideration two phenomena. First, after decomposing our domain into small  cells and working in a  small  cell near the domain's boundary, we have to express the integrals in  a flat  geometry, which  requires  a careful expansion   of the  Riemannian metric  in particular. This part is  essentially  the same as for the constant magnetic field case in \cite{HelMo4}.

Then, we have to express the magnetic potential in adapted coordinates, in each small cell, and apply a Taylor expansion  and  a  gauge transformation to obtain a ``normal'' form, i.e. a simpler effective magnetic potential. In  this part, we deviate from the  constant magnetic field situation and find additional terms in the effective magnetic potential. Interestingly, we can still show that the analysis with  this magnetic potential is somehow independent of  those additional terms and treat the new model  as a perturbation of the  model with a constant magnetic field.

The paper is organized as follows. In Section~\ref{sec:coordinates} we introduce the adapted coordinates in  a small  ``boundary'' cell. In Section~\ref{sec:helix}, we analyze the case of the unit ball  with the ``helical'' magnetic field occurring in  liquid crystals and verify that  Assumptions~\ref{ass:C1}  and \ref{ass:C2} hold.  Interestingly, after computing the energy in \eqref{eq:hat-gam}, we notice that this example shows  a  phenomenon of  multiple ``wells'' induced by the ``magnetic'' geometry. 

In Section~\ref{sec:1Dmod}, we review two standard 1D operators that we need in defining the quantities appearing in \eqref{eq:hat-gam} and the statement in  Theorem~\ref{thm:main}. Then, in Section~\ref{sec:Nmod}, we introduce a new model, specific to our case of a varying magnetic  field with a constant intensity, and analyze it through a perturbation argument. 

With the model in Section~\ref{sec:Nmod}, we can adjust  the proof in \cite{HelMo4} and prove Theorem~\ref{thm:main}. The first step is to localize the ground states near the boundary, which is the  content of Section~\ref{sec:loc}. Then, the approximation of the quadratic form  and the magnetic potential  are the subject of Section~\ref{sec:qf}, which allows us, in the subsequent Section~\ref{sec:lb}, to obtain a lower bound on  the lowest eigenvalue.

Finally, Section~\ref{sec:ub} is devoted to the computation of the energy of a trial state, which yields an upper  bound of the lowest eigenvalue, and thereby completes  the proof of  Theorem~\ref{thm:main}.
 
\section{Adapted coordinates}\label{sec:coordinates}
We recall a rather standard choice of coordinates in the neighborhood of $\Gamma$.
\subsection{Description of  the coordinates}\label{sbsec:coord}

Let $g_0$ be the Riemannian metric on $\R^3$, which induces a Riemanian  metric $G$ on $\partial\Omega$. Given two  vector fields $\mathbf X,\mathbf Y$ of $\R^3$, we denote by
\begin{equation}\label{eq:not-metric}
\mathbf X\cdot\mathbf Y=\langle \mathbf X,\mathbf Y\rangle:=g_0(\mathbf X,\mathbf Y)\,.
\end{equation}
Consider a direct frame  $(\mathbf V(x),\mathbf T(x),\mathbf  N(x) )_{x\in\Gamma}$ along $\Gamma$ such that
\begin{itemize}
\item $\mathbf T(x)$ is an oriented unit tangent vector of $\Gamma$\,;
\item $\mathbf V(x):=\mathbf T(x)\times \mathbf N(x)$,  hence determining an oriented normal to the curve $\Gamma$ in the tangent space to $\partial \Omega$.
\end{itemize}
For $m\in\Gamma$, let $\Lambda_{m}$ be the geodesic that passes through $m$  and is normal to $\Gamma$. 
Let $x_0\in\Gamma$. In some neighborhood $\mathcal N_{x_0}\subset\overline\Omega$ of  $x_0$, we can introduce new coordinates $(r,s,t)$ as follows:
\begin{itemize}
\item For $x\in\mathcal N_{x_0}$, $p(x)\in\partial\Omega$ is defined by  ${\rm dist}(x,p(x))=t(x):={\rm dist}(x,\partial\Omega)$;
\item For $x\in\mathcal N_{x_0}$,  $\gamma(x)\in\Gamma$ is defined by ${\rm dist}_{\partial\Omega}(p(x),\gamma(x))=\dist_{\partial\Omega}(p(x),\Gamma)$, where ${\rm dist}_{\partial\Omega}$  denotes the (geodesic) distance in $\partial\Omega$\,;
\item $\Gamma$ is parameterized by arc-length $s$  so that $s=s_0$ defines  $x_0$, and for $x\in\mathcal N_{x_0}$,  $s=s(x)$ defines $\gamma(x)$\,;
\item For $x\in\mathcal N_{x_0}$, the geodesic $\Lambda_{p(x)}$ passing through $p(x)$ is  parameterized by arclength  $r$, so  that $r=0$   defines $\gamma(x)$ and $r=r(x)$ defines $p(x)$.
\end{itemize} 
In this way, we observe that $\Phi_{x_0}$
\begin{equation}\label{eq:Phi-x0}
 \mathcal N_{x_0}\ni  x\mapsto \Phi_{x_0} (x):= (r(x),s(x),t(x))\in \R\times\R\times\R_+
\end{equation} is a local diffeomorphism. Thus, we can  pick a  sufficiently small $\epsilon_0>0$ such  that
\begin{equation}\label{eq:Phi-1-x0} 
(r,s,t)\in (-\epsilon_0,\epsilon_0)\times(-\epsilon_0+s_0,s_0+\epsilon_0)\times(0,\epsilon_0)
\to x=\Phi_{x_0}^{-1}(r,s,t)
\end{equation}
is a diffeomorphism, whose image is   a neighborhood of $x_0\in\Gamma$ 
parameterized by $(r,s,t)$. Within these coordinates, $t=0$ means that we are on $\partial\Omega$, and $r=t=0$ means we are on the curve $\Gamma$. We can then compute 
\begin{equation}\label{eq:kappa-n-B}
|d^T(\mathbf B\cdot \mathbf N) (x)|=|\partial_r(\mathbf B\cdot\mathbf N)|_{r=0} |\quad(x\in\Gamma)\,.
\end{equation} 
It is convenient to express  the magnetic field along $\Gamma$ as follows
\begin{equation}\label{eq:B-r,t=0}
\Bb(x)= \sin\theta \, \Tb(x) + \cos\theta \, \Vb(x)\quad\big(x=\Phi_{x_0}^{-1}(0,s,0)\in\Gamma\big),
\end{equation}
where  $\theta=\theta(s)\in[-\frac\pi2,\frac\pi2]$  is  the angle defined by
\begin{equation}\label{eq:theta-r,t=0}
\theta=\arcsin\big(\Bb(x)\cdot\Tb(x)\big).
\end{equation}
\subsection{The metric in the new coordinates}\label{sbsec:metric}

Let us  consider an arbitrary point $x_0\in\Gamma$ and a neighborhood $\mathcal N_{x_0}\subset\overline{\Omega}$ of $x_0$ such that the adapted coordinates introduced in \eqref{eq:Phi-x0} and \eqref{eq:Phi-1-x0}  are valid. Modulo a translation, we can center the coordinates at  $x_0$ so that $(r=0,s=0,t=0)$ are the coordinates of $x_0$ in the new frame.  In the sequel, we follow closely the presentation of \cite[Sec.~8]{HelMo4} mainly following  the first chapter of \cite{DHKW} (see also 
the volume two of  Spivak's  book \cite{Sp}).

We label the new coordinates as follows
\begin{equation}\label{eq:label(r,s,t)}
(y_1,y_2,y_3)=(r,s,t)\,,
\end{equation}
and  the Riemanian metric $g_0$ becomes \cite[Eq. (8.26)]{HelMo4}
\begin{equation}\label{eq:g0-y}
g_0=dy_3\otimes dy_3+\sum_{1\leq i,j\leq 2}\big[G_{ij}-2y_3K_{ij}+y_3^2L_{ij} \big]dy_{i}\otimes dy_j
\end{equation}
where: 
\begin{itemize}
\item $G:= \sum\limits_{1\leq i,j\leq 2}G_{ij} dy_{i}\otimes dy_j$ is the first fundamental form on $\partial\Omega$\,;
\item $K:= \sum\limits_{1\leq i,j\leq 2}K_{ij} dy_{i}\otimes dy_j$ is the second fundamental form on $\partial\Omega$\,;
\item  $L:= \sum\limits_{1\leq i,j\leq 2}L_{ij} dy_{i}\otimes dy_j$ is the third fundamental form on $\partial\Omega$\,.
\end{itemize}
 The matrix $g$ of  the metric $g_0$ takes the form 
\begin{equation}\label{eq:g-ij}
g:=(g_{ij})_{1\leq i,j\leq 3}=\left(\begin{array}{lll} g_{11}&g_{12}&0\\
g_{21}&g_{22}&0\\
0&0&1 \end{array}\right)
\end{equation}
whose inverse is
\begin{equation}\label{eq:g-ij-inverse}
g^{-1}=(g^{ij})_{1\leq i,j\leq 3}=\left(\begin{array}{lll} g^{11}&g^{12}&0\\
g^{21}&g^{22}&0\\
0&0&1 \end{array}\right)\,.
\end{equation}
We will  express these matrices in a more pleasant form involving, in particular, the curvatures on the boundary. To that end, let $s\mapsto \gamma(s)$ be an arc-length parameterization of $\Gamma$ near $x_0$, so that $|\dot\gamma(s)|=1$, $\gamma(0)=x_0$ and $\mathbf T(\gamma(s))=\dot\gamma (s)$. We can introduce the \emph{geodesic} and \emph{normal} curvatures at $\gamma(s)$,  $\kappa_g(\gamma(s))$ and $\kappa_n(\gamma(s))$,  as follows 
\begin{equation}\label{eq:kg-kn}
\ddot\gamma(s)=-\kappa_g(\gamma(s))\mathbf V(\gamma(s)) +\kappa_n(\gamma(s))\mathbf  N(\gamma(s))\,.
\end{equation}
The choice of our coordinates $(r,s)$ ensures that  the metric $G$ is diagonal  on $\partial\Omega$ \cite[Lem.~8.2]{HelMo4}
\begin{equation}\label{eq:alpha(r,s)}
G=dr\otimes dr+\alpha(r,s)ds\otimes ds,
\end{equation}
with 
\begin{equation}\label{eq:alpha(r,s)*}
\alpha (r,s) = 1- 2\kappa_g\big(\gamma(s)\big)  r +\Og(r^2)\;,\quad\alpha(0,s)=1\,,
\end{equation}
and
\begin{equation}\label{eq:alpha(r,s)**}
\frac{\partial\alpha}{\partial s} (0,s) = 0\;.
\end{equation}
Then, with \eqref{eq:label(r,s,t)}, we have  for the determinant of the matrix of $g$  (see \cite[Eq. (8.29) \&  (8.30)]{HelMo4}),
\begin{equation}\label{eq:det-g}
|g|=\alpha(r,s)-2t\big[ \alpha(r,s)K_{11}(r,s) +K_{22}(r,s)\big]+t^2\varepsilon_3(r,s,t)\,,
\end{equation}
and
\begin{equation}\label{eq:inv-g}
(g^{ij})_{1\leq i,j\leq 2}=\left(\begin{array}{ll}  1&0\\
0&\alpha^{-1}(r,s) \end{array}\right)+2t \left(\begin{array}{ll}  K_{11}(r,s)&\alpha^{-1} K_{12}(r,s)\\
\alpha^{-1} K_{21}(r,s)&\alpha^{-2}K_{22}(r,s) \end{array}\right)+t^2R\,,
\end{equation}
where $\varepsilon_3$ and $R$ are smooth functions.

\subsection{The operator and quadratic form}\label{sbsec:op-qf}

We continue to work in the setting of Subsection~\ref{sbsec:metric} . We introduce the following neighborhood of $x_0$
\begin{equation}\label{eq:V-x0}
V_{x_0}=\Phi_{x_0}^{-1}(\tilde V_{x_0})\,,
\end{equation} 
where (recall \eqref{eq:label(r,s,t)})
\begin{equation}\label{eq:tilde-V-x0}
\tilde V_{x_0}=  \{(y_1,y_2,y_3)\in (-\epsilon_0,\epsilon_0)\times(-\epsilon_0,\epsilon_0)\times(0,\epsilon_0)\}\,.
\end{equation}
Given a function $u:V_{x_0}\to\C$, we assign to it the function $\tilde u:V_{x_0}\to\C$ defined by
\begin{equation}\label{eq:tilde-u}
\tilde u(y_1,y_2,y_3)=u(\Phi_{x_0}^{-1}(y_1,y_2,y_3))\,.
\end{equation}
By  the considerations in Subsection~\ref{sbsec:metric}  on  the Riemanian metric, if $u\in L^2(V_{x_0},dx)$, then $\tilde u
\in L^2(\tilde V_{x_0}, |g|^{1/2}dy)$ and
\begin{equation}\label{eq:norm-g}
\int_{V_{x_0}}|u(x)|^2dx=\int_{\tilde V_{x_0}}|\tilde u(y)|^2|g|^{1/2}dy\,.
\end{equation}
Moreover, assuming $u$ supported in $V_{x_0}$, we have the quadratic form formula \cite[Eq. (8.27)]{HelMo4} 
\begin{equation}\label{eq:qf-y}
\begin{aligned}
q_{\Ab}^h(u)&:=\int_{V_{x_0}}|(h\nabla-i\Ab)u|^2dx\\
&=\int_{\tilde V_{x_0}} \Big[|(hD_{y_3}-\tilde A_3)\tilde u|^2+\sum_{1\leq i,j\leq 2}g^{ij}(hD_{y_i}-\tilde A_i)\tilde u\cdot\overline{(hD_{y_j}-\tilde A_j)\tilde u} \Big]|\, g|^{1/2}dy
\end{aligned}
\end{equation}
where the new magnetic  potential $\tilde\Ab=(\tilde A_1,\tilde A_2,\tilde A_3)$ is assigned to $\Ab=(A_1,A_2,A_3)$ by the relation
\begin{equation}\label{eq:tA}
A_1dx_1+A_2dx_2+A_3dx_3=\tilde A_1dy_1+\tilde A_2 dy_2+\tilde A_3dy_3\,,
\end{equation}
and after performing a (local) gauge transformation,  we may  assume that 
\begin{equation}\label{eq:A3=0}
\tilde A_3=0\,.
\end{equation} 
The operator $P_{\Ab}^h$ in \eqref{eq:op} can be expressed in the new coordinates as follows \cite[Eq.~(8.28)]{HelMo4}
\begin{multline}\label{eq:op-y}
P_{\Ab}^h=(hD_{y_3}-\tilde A_3)^2+\frac{h}{2i}|g|^{-1}
\frac{\partial}{\partial y_3}
|g|(hD_{y_3} - \tilde A_3)\\
+ |g|^{-1/2}
\sum_{1\leq i,j\leq 2}
(hD_{y_j} 
- \tilde A_j)|g|^{1/2}g^{ij}   (hD_{y_i} 
- \tilde A_i)\,.
\end{multline}
\section{Helical magnetic fields}\label{sec:helix}
\subsection{Preliminaries}
Let $\tau>0$ and consider the magnetic potential
\begin{equation}\label{eq:N-tau}
 \Ab(x)=\nb_\tau(x):=\Big(\frac1\tau\cos(\tau x_3),\frac1\tau\sin(\tau x_3),0\Big)\,,
\end{equation}
which generates the  magnetic field
\begin{equation}\label{eq:curl-N-tau}
\Bb(x)=\curl\Ab(x)=-\tau \Ab(x)
\end{equation}
with constant intensity
\begin{equation}\label{eq:|B|=tau}
 |\Bb(x)|= 1\,.
\end{equation}
We will verify  that Assumptions C1-C2 hold for this particular magnetic field in  the case where $\Omega$ is the unit ball.  In particular, with in mind that 
 $\widehat\gamma_{0,\Bb}$ and $\widetilde\gamma_{0,\Bb}$ are introduced in \eqref{eq:hat-gam} and \eqref{eq:tilde-gam} respectively and that  $\delta_0\in]0,1[$ and  $\widehat \nu_0>0$  will be introduced in \eqref{eq:Th0} and in \eqref{eq:Mont} (there is no need in this subsection to know more about them) 
  we will find that
\begin{equation}\label{eq:min-eff-en}
\widehat{\gamma}_{0,\Bb}=2^{-2/3} \widehat{\nu}_0 \delta_0^{1/3} C(\tau,\delta_0) \,,
\end{equation}
and for $\tau\leq \tau_0$, the equality,  
\begin{equation}\label{eq:min-eff-en*}
\{x\in\Gamma~|~ \widetilde{\gamma}_{0,\Bb}(x)=2^{-2/3} \widehat{\nu}_0 \delta_0^{1/3} \}=\{(0,\pm 1,0)\}\,,
\end{equation}
   where $\tau_0$ is a  constant and $C(\tau,\delta_0)$ is explicitly computed (see Proposition~\ref{prop:min-example}).
 
The inward normal of $\Omega=\{x\in\R^3~|~|x|<1\}$ along $\partial\Omega$ is \begin{equation}\label{eq:N}
\Nb(x)=-x\quad(|x|=1)\,.
\end{equation}
 The restriction of the magnetic field $\Bb$ to the boundary is then tangent to $\partial\Omega$ on the following set
\begin{equation}\label{eq:Gam}  
\Gamma=\{x\in\partial\Omega~|~x\cdot \Ab(x)=0\}\,.
\end{equation}

\subsection{$\Gamma$  is a regular curve}
For $|x|=1$, the equation $x\cdot \Ab(x)=0$ reads as follows
\begin{equation}\label{eq:Gam*} 
x_1\cos(\tau x_3)+x_2\sin(\tau x_3)=0\,.
\end{equation}

\begin{proposition}\label{prop:Gam-rc}
The set $\Gamma$ introduced in \eqref{eq:Gam} is a $C^\infty$ regular curve.
\end{proposition}

\begin{proof}
 The proof follows by  constructing  an atlas on $\Gamma$, 
\[\{\big(\mathbf c_i,U:=(-1,1)\big),~1\leq i\leq 4 \}\]
which turns $\Gamma$ to a $C^\infty$ regular curve.

 Let us introduce the charts $(\mathbf c_1,U)$ and $(\mathbf c_2,U)$ which cover $\Gamma\setminus\{(0,0,\pm1)\}$. These charts are obtained by expressing  $x_1$  and  $x_2$ in \eqref{eq:Gam*}   in terms of $x_3\in(-1,1)$, provided that $(x_1,x_2,x_3)\not=(0,0,\pm1)$.
We write  for $\alpha \in ]-\pi,\pi]$  
\begin{equation}\label{eq:coord}
x_1=\sqrt{1-x_3^2}\, \cos \alpha,\quad x_2= \sqrt{1-x_3^2}\, \sin \alpha\,.
\end{equation}
Then \eqref{eq:Gam*} becomes, for $x_3^2<1$, 
\begin{equation}\label{eq:cos=0}
 \cos(\tau x_3-\alpha )=0
 \end{equation}
which in turn yields
\[\alpha=\tau x_3-\frac\pi2+k\pi,\quad k\in\mathbb Z\,.\]
In  this way, we get two branches of $\Gamma$  parameterized by $x_3$ and defined as follows 
\[ x_3\in(-1,1)\mapsto\mathbf c_1(x_3):=\left(\begin{array}{c} x_1=\sqrt{1-x_3^2}\, \sin(\tau x_3)\\ x_2=-\sqrt{1-x_3^2}\,\cos(\tau x_3)\\  x_3\end{array}\right)\] 
and
\[x_3\in(-1,1)\mapsto\mathbf c_2(x_3):=\left(\begin{array}{c}x_1=-\sqrt{1-x_3^2}\sin(\tau x_3)\\ x_2=\sqrt{1-x_3^2}\cos(\tau x_3)\\  x_3\end{array}\right)\,.\] 
Both of the foregoing branches represent regular curves.  Furthermore, $\mathbf c_{1}$ and $\mathbf c_2$ can be extended \emph{by continuity} to  the interval $[-1,1]$, yielding a  continuous representation of all $\Gamma$. 

Now  we introduce the charts $(\mathbf c_3,U)$ and $\mathbf c_4,U)$ that cover the points $(0,0,\pm 1)$.  In a neighborhood of $(x_1,x_2,x_3)=(0,0,\pm 1)$,   we parameterize   a branch of $\Gamma$  with respect to $\rho :=\sqrt{x_1^2+x_2^2}$  as follows
\[x_1=\rho\cos \alpha \,,~x_2=\rho\sin \alpha\, ,~x_3=\sqrt{1-\rho^2}\,.\]
With   this in hand,  \eqref{eq:cos=0} continues to hold for $x_3\not=0$ and we can write again  $\alpha=\tau x_3-\frac\pi2+k\pi$ for some $k\in\mathbb Z$.  Consequently, we get  two regular branches of $\Gamma$ defined as follows
\[\rho\in(-1,1)\mapsto\mathbf c_3(\rho):=\left(\begin{array}{rl}  x_1&=\rho\sin(\tau\sqrt{1-\rho^2})\\ x_2&= -\rho\cos(\tau\sqrt{1-\rho^2})\\ x_3&=\sqrt{1-\rho^2}\end{array}\right)\]
and
\[\rho\in(-1,1)\mapsto\mathbf c_4(\rho):=\left(\begin{array}{rl}  x_1&=-\rho\sin(\tau\sqrt{1-\rho^2})\\ x_2&= \rho\cos(\tau\sqrt{1-\rho^2})\\ x_3&=\sqrt{1-\rho^2}\end{array}\right)\,.\] 
\end{proof}

\subsection{Explicit formulas in adapted coordinates}

Note  that $\mathbf c:=\mathbf c_1$ and $\mathbf c_2$ parameterize all of  $ \Gamma\setminus\{(0,0,\pm1)\}$. 
 By symmetry considerations, we will compute, on $\mathbf c\big(\,(-1,1)\,\big)$ only, 
\begin{equation}\label{eq:d(BN)}
 |d^T(\mathbf B\cdot \mathbf N)|=\tau |d^T(\mathbf A\cdot \mathbf N)|\quad{\rm and}\quad   |\mathbf B\cdot \mathbf T|=\tau  |\mathbf A\cdot \mathbf T|\,.
\end{equation}
First we note that $\mathbf N=-x$ on $\partial\Omega$ and introduce the arc-length parameter 
\[s(x_3) =\int_0^{x_3}|\mathbf c'(\tilde x_3)|d\tilde x_3\] of $x_3\mapsto\mathbf c(x_3)$, which satisfies
\begin{equation}\label{eq:ds/dx}
 s'(x_3)=|\mathbf c'(x_3)|=\sqrt{\frac{1+\tau^2(1-x_3^2)^2}{1-x_3^2}}\,.
\end{equation}
 Clearly, $x_3\in(-1,1)$ can be expressed in terms of the arc-length parameter as $x_3=x_3(s)$  
with 
\begin{equation}\label{eq:dx/ds}
 m(x_3):=\frac{dx_3}{ds}(s(x_3)) = \sqrt{\frac{1-x_3^2}{1+\tau^2(1-x_3^2)^2}}\,.
 \end{equation}
 The arc-length parameterization is now given by
 \begin{equation}\label{eq:c(s)}
\gamma(s):=\mathbf c( x_3(s) )\,,
\end{equation}
and consequently,  with $\mathbf c=\mathbf c_1$, we have
\begin{equation}\label{eq:N,T-c}
\mathbf N(\gamma(s))=-\mathbf \gamma(s)=\left(\begin{array}{c}- \sqrt{1-x_3^2}\sin(\tau x_3)\\ \sqrt{1-x_3^2}\cos(\tau x_3)\\  - x_3\end{array}\right) ~{\rm with ~}x_3=x_3(s)\,,\end{equation}
 and
\begin{align*} \mathbf T(\gamma(s))&=\frac{d}{ds}\gamma(s)=\left(\begin{array}{c}T_1\\T_2\\T_3\end{array}\right)\\
&=m(x_3)\left(\begin{array}{c}-\frac{x_3\sin(\tau x_3)}{\sqrt{1-x_3^2}} +\tau\sqrt{1-x_3^2}\cos(\tau x_3)\\  \frac{x_3\cos(\tau x_3)}{\sqrt{1-x_3^2}} +\tau\sqrt{1-x_3^2}\sin(\tau x_3)\\  1\end{array}\right) \,.
\end{align*}
We also introduce the normal  vector to $\Gamma$ on $\gamma(s)$,
 \begin{align*}
 \mathbf V(\gamma(s))&=\mathbf T(\gamma(s))\times\mathbf N(\gamma(s))=\left(\begin{array}{c}V_1\\V_2\\V_3\end{array}\right)\\
 &=  m(x_3)\left(\begin{array}{c}-\frac{x_3^2\cos(\tau x_3)}{\sqrt{1-x_3^2}} -\tau x_3\sqrt{1-x_3^2}\sin(\tau x_3) -\sqrt{1-x_3^2}\cos(\tau x_3)\\  -\frac{x_3^2\sin(\tau x_3)}{\sqrt{1-x_3^2}} +\tau x_3\sqrt{1-x_3^2}\cos(\tau x_3) -\sqrt{1-x_3^2}\sin(\tau x_3)\\  \tau(1-x_3^2)\end{array}\right) \,.
\end{align*}
We are  now ready to prove that our  magnetic field $\Bb$ verifies the condition (C2) appearing in Assumption~\ref{ass:C2}  
\begin{proposition}\label{prop:C2}
Let  $\Bb$ be the magnetic field introduced in \eqref{eq:curl-N-tau}. For all $x\in \Gamma$, we have 
\[|\Bb(x)\cdot\mathbf T(x)|= \frac{\tau(1-x_3^2)}{\sqrt{1+\tau^2(1-x_3^2)^2}}\,.\] In particular, $\Bb$   satisfies the condition $\rm(C2)$.
\end{proposition}
\begin{proof}
It is straightforward to compute
\begin{equation}\label{eq:A.N}
|\mathbf A(x)\cdot\mathbf T(x)| = \frac {1}{\tau} (|\cos(\tau x_3) T_1+\sin(\tau x_3)T_2|)= \frac{1-x_3^2}{\sqrt{1+\tau^2(1-x_3^2)^2}}\,,
\end{equation}
which holds for all $-1\leq x_3 <  1$ and $x=\mathbf c(x_3)$.  Similarly, we can compute $|\mathbf A(x)\cdot\mathbf T(x)|$ for all $x=\mathbf c_2(x_3)\in\Gamma$,  and get that \eqref{eq:A.N} holds globally on $\Gamma$,   since $\Gamma$ is a regular curve. Finally, $\Bb(x)$ is orthogonal to $\mathbf T(x)$ if and only if $x_3^2=1$, thereby  (C2) holds. 
\end{proof}

 Our next task is to show that our magnetic field satisfies the condition (C1) in Assumption~\ref{ass:C1}.
 \begin{proposition}\label{prop:C1}
Let  $\Bb$ be the magnetic field introduced in \eqref{eq:curl-N-tau}. For all $x\in \Gamma$, we have 
\begin{equation} \label{eq:aa5*} 
 \kappa_{n,{\mathbf B}} (x) = \sqrt{1+\tau^2(1-x_3(s)^2)^2}\,.
\end{equation}
 In particular, $\Bb$   satisfies the condition $\rm(C1)$.
\end{proposition}
\begin{proof}
 By Proposition~\ref{prop:Gam-rc}, $\Gamma$ is  a regular curve. So all we need to verify that $\Bb$ satisfies $\rm(C1)$, is to  derive \eqref{eq:aa5*} and observe that it yields $\kappa_{n,{\mathbf B}} (x) \not=0$ everywhere on $\Gamma$.

 Consider $x=\mathbf c_1(x_3)$ with $x_3=x_3(s)$, i.e. $x=\gamma(s)$. At the point $\gamma(s)$, the geodesic $\Lambda_{\gamma(s)}$  normal to the curve  $\Gamma$ is  the  great circle  (of center $0$ and radius $1$) in the $(\mathbf V(\gamma(s)),\mathbf N(\gamma(s))$ plane. A point $P=P(r,s)$ on  $\Lambda_{\gamma(s)}$ can be described by  the corresponding vector $\mathbf p(r,s)=\overrightarrow{OP}$ as follows
\[ \mathbf p(r,s) =-\cos r\, \mathbf N(\gamma(s))-\sin r\, \mathbf V(\gamma(s))\,, \]
where $r$ is the  angle between $\mathbf p$ and $-\mathbf N$; hence $r$ is an arc-length length parameter of $\Lambda_{\gamma(s)}$, and  for $r=0$, $p(r,s)=\gamma(s)$.
\begin{figure}[t]
\begin{center}
\includegraphics[scale=3]{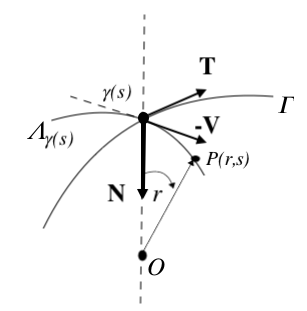}
\end{center}
\caption{The curve $\Gamma$ and the geodesic $\Lambda_{\gamma(s)}$ passing through $\gamma(s)$.}\label{fig1} 
\end{figure}
Now, we can introduce the coordinates $(r,s,t)$ in a neighborhood of $\gamma(s_0)$ as follows (see Fig.~\ref{fig1})
\begin{equation}\label{eq:x(r,s,t)-ball}
x(r,s,t)= -(\cos r+t)\mathbf N(\gamma(s))-\sin r\mathbf V(\gamma(s))\,.
\end{equation}
 For $x=\gamma(s)$, we would like to compute
$\kappa_{n,\Bb}(x)=|d^T(\Bb\cdot\Nb )|$.  We will show that $\kappa_{n,\Bb}(x)= \partial_r(\Bb\cdot\Nb)|_{r=t=0}$ and end up with the computation of 
$ \big|\partial_r(\Bb\cdot\Nb)|_{r=t=0}\,.$

 Notice that, by \eqref{eq:N,T-c}, we have 
\begin{equation*}
\begin{array}{ll}x_3(r,s,t)& = -(\cos r+t)\mathbf N_3(\gamma(s))-\sin r\mathbf V_3(\gamma(s))\\ 
&=  (\cos r+t) x_3(s) -\sin r \, m(x_3(s)) \tau (1-x_3(s)^2) \,,
\end{array}
\end{equation*}
and  we observe that by \eqref{eq:x(r,s,t)-ball},
\begin{equation} \label{eq:aa1}  \frac{\partial x}{\partial r}\,\big|_{r=t=0}= -V(\gamma(s))\,.
\end{equation}
In particular we have
\[
\frac{\partial x_3}{\partial r}\,\big|_{r=t=0}  =- \tau(1-x_3^2(s)) m(x_3(s))\,.
\]
Now, using \eqref{eq:dx/ds} and \eqref{eq:N,T-c},  we get from \eqref{eq:N-tau}  that 
\begin{equation}\label{eq:aa2}
\frac{\partial\Ab}{\partial r}\cdot\mathbf N\,\big|_{r=t=0}= - 
\frac{\tau(1-x_3^2)^2}{\sqrt{1+\tau^2(1-x_3^2)^2}} \,. 
\end{equation}
Moreover, by \eqref{eq:aa1}  we have
\[ \frac{\partial}{\partial r}\mathbf N(x(r,s,t))\,\big|_{r=t=0}=\mathbf V(\gamma(s))\]
and
\begin{equation}\label{eq:aa3} \begin{aligned}\mathbf A\cdot \frac{\partial}{\partial r}\mathbf N(x(r,s,t))\,\big|_{r=t=0}&=\frac1\tau \cos(\tau x_3(s))V_1+\frac1\tau\sin(\tau x_3(s))V_2\\
&= - \frac{1}{\tau\sqrt{1+\tau^2(1-x_3^2)^2}}\,.\end{aligned}
\end{equation}
Summing up, we deduce from \eqref{eq:aa2} and \eqref{eq:aa3} that
\begin{equation}\label{eq:dr-A.N} 
\big|\partial_r ( \mathbf A\cdot\mathbf N)\,|_{r=t=0}\big|=
\frac{1}{\tau} \sqrt{1+\tau^2(1-x_3^2)^2}\,.
\end{equation}
We also observe that  $\partial_s  ( \mathbf A\cdot \mathbf N)\,|_{r=t=0}=0$ and we get 
\begin{equation} \label{eq:aa4}  
| d^T (\mathbf A\cdot \mathbf N) | (\gamma(s)) = \frac{1}{\tau} \sqrt{1+\tau^2(1-x_3^2)^2} 
\,,
\end{equation}
on each branch (including the end points).  Inserting this into \eqref{eq:d(BN)}, we  get the identity in \eqref{eq:aa5*}.
\end{proof}

We return to the function in \eqref{eq:tilde-gam}  and can give its expression in coordinates. We  deduce from \eqref{eq:A.N} and  \eqref{eq:aa5*}:
\[\widetilde\gamma_{0,\Bb}(x)=2^{-2/3} \widehat{\nu}_0 \delta_0^{1/3} \left(1+ \tau^2(1-x_3^2)^2\right)^{1/3}\left(1-(1-\delta_0)\frac{\tau(1-x_3^2)}{\sqrt{1+\tau^2(1-x_3^2)^2}}\right)^{1/3} \]
for all $x= (\pm \sqrt{1-x_3^2}\sin(\tau x_3),\mp\sqrt{1-x_3^2}\cos(\tau x_3),x_3)$ with $-1\leq x_3\leq 1$.

 Consequently, we can compute the quantity appearing in the two terms asymptotics by computing $\inf_{x\in\Gamma}\widetilde\gamma_{0,\Bb}(x)$ and determining where the infimum is attained.
 
\begin{proposition}\label{prop:min-example}
Let 
\[ \tau_0= \frac1{\sqrt{2}}\left(\frac{1}{\sqrt{\delta_0+\delta_0(1-\delta_0)}}-1\right)^{1/2}\,. \]
The following holds:
\begin{enumerate}
\item If $0<\tau\leq \tau_0$, then 
\[\inf_{x\in\Gamma}\widetilde\gamma_{0,\Bb}(x)=2^{-2/3} \widehat{\nu}_0 \delta_0^{1/3}(1+\tau^2)^{1/3}\Big(1-(1-\delta_0)\frac{\tau^{1/3}}{(1+\tau^2)^{1/6}} \Big)=\widetilde\gamma_{0,\Bb}(0,\pm1,0)\,.\]
\item If $\tau>\tau_0$, then
\[\inf_{x\in\Gamma}\widetilde\gamma_{0,\Bb}(x)=2^{-2/3} \widehat{\nu}_0 \delta_0^{1/3}(1+\tau_0^2)^{1/3}\Big(1-(1-\delta_0)\frac{\tau_0^{1/3}}{(1+\tau_0^2)^{1/6}} \Big)\,,\]
and the minimum is attained on the points
\[\Big(\,\pm\sqrt{\frac{\tau_0}{\tau}}\sin\tau\sqrt{1-\frac{\tau_0}{\tau}}, \mp\sqrt{\frac{\tau_0}{\tau}}\cos\tau\sqrt{1-\frac{\tau_0}{\tau}}, \sqrt{1-\frac{\tau_0}{\tau}}\,\Big)\]
and
\[\Big(\,\pm\sqrt{\frac{\tau_0}{\tau}}\sin\tau\sqrt{1-\frac{\tau_0}{\tau}}, \pm\sqrt{\frac{\tau_0}{\tau}}\cos\tau\sqrt{1-\frac{\tau_0}{\tau}}, -\sqrt{1-\frac{\tau_0}{\tau}}\,\Big)\,.\]
\end{enumerate}
\end{proposition} 
\begin{remark}\label{rem:tunneling}
 In  the case where  $\Omega=B(0,1)$ is the unit ball and  the magnetic field is constant, $\Bb=(0,0,1)$, we have  $\Gamma=\{x_1^2+x_2^2=1,~x_3=0\}$  and $\widetilde\gamma_{0,\Bb}(x)$ is constant on $\Gamma$.    Proposition~\ref{prop:min-example} shows a quite different phenomenon when only the  intensity  of $\Bb$ is constant, $|\Bb|=1$. In fact, $\widetilde\gamma_{0,\Bb}(x)$ is no more constant along $\Gamma$ and may have two symmetric minimum points, $(0,\pm1,0)$, which is  the signature of an interesting double well tunnel effect {\rm \cite{HeSj1}} related to  the  magnetic geometry  of the problem.
\end{remark}
\begin{proof}[Proof of Proposition~\ref{prop:min-example}]
Let us introduce $v=\tau(1-x_3^2)\in[0,\tau]$ and $\mu_0=1-\delta_0\in(0,1)$. Then
\[\widetilde\gamma_{0,\Bb}(x)=2^{-2/3} \widehat{\nu}_0 \delta_0^{1/3}\big(f(v)\big)^{1/3}\]
where $$ f(v)=1+v^2-\mu_0 v \sqrt{1+v^2}\,.$$
We have to  minimize $f(v)$ on $[0,\tau]$. Notice that
\[f'(v)=2v-\mu_0\frac{1+2v^2}{\sqrt{1+v^2}}\,,\]
and the equation $ f'(v)=0$ has a  unique positive solution, which is the solution of
\[v^4+v^2=\frac{\mu_0^2}{4(1-\mu_0^2)}\,. \]
This solution is given by
$$ 
\tau_0 = \frac{1}{\sqrt{2}} \frac{\mu_0} { \sqrt{ 1 + \sqrt{1-\mu_0^2} } \sqrt{1-\mu_0^2} }
$$
and observe that $f'(v)<0$ for $0<v<\tau_0$ and $f'(v)> 0$ for  $v>\tau_0$.  Then, for $\tau\leq\tau_0\,$, 
\[\min_{v\in [0,\tau]}f(v)=f(\tau)\,,  \]
while for $\tau>\tau_0\,$,
\[\min_{v\in[0,\tau]}f(v)=f(\tau_0)\,.\]
\end{proof}

\section{1D Models}\label{sec:1Dmod}
The aim of this section is to recall the now  standard properties of two important models.
\subsection{The de\,Gennes model}
We refer to \cite{DaHe,HelMo2} for the proof of these now standard properties which are presented below. 
For $\xi\in\R$, we consider the  harmonic oscillator  on $\R_+$:
\begin{equation}\label{eq:Harm-osc}
H(\xi):=D_t^2+(t-\xi)^2\,,
\end{equation}
with Neumann boundary condition at $0$. We denote by $\mu(\xi)$ its lowest eigenvalue.  $\xi \mapsto \mu(\xi)$ admits a unique minimum at a point $\xi_0$ which in addition is non-degenerate. This leads to introduce the spectral constants, $\Theta_0$  and $\delta_0$:
\begin{equation}\label{eq:Th0}
\Theta_0=\inf_{\xi\in\R}\mu(\xi)=\mu(\xi_0),\quad \delta_0=\mu''(\xi_0)\,,
\end{equation}
where $\xi_0=\sqrt{\Theta_0}$.\\
 Moreover  $\frac12<\Theta_0<1$ and that  $0< \delta_0< 1$.
 $\Theta_0$  is called the de\,Gennes constant.\\
If $\varphi_0\in L^2(\R_+)$ denotes  the positive and normalized ground state of $H(\xi_0)$, 
\begin{equation}\label{eq:mu'=0}
\int_{\R_+}(t-\xi_0)|\varphi_0(t)|^2dt=0\,, 
\end{equation}
which  amounts to  saying, via the Feynman-Hellmann formula,  that $\mu'(\xi_0)=0$. We also introduce the regularized resolvent $\mathcal R_0\in\mathcal L(L^2(\R_+))$ as follows
\begin{equation}\label{eq:Reg-res}
\mathcal R_0u=\begin{cases}
\big(H(\xi_0)-\Theta_0\big)^{-1}u&{\rm if}~u\perp \varphi_0\\
0&{\rm if}~u\parallel \varphi_0
\end{cases}\,.
\end{equation}

\subsection{The Montgomery model} 
Here we refer to \cite{HelMo1} and \cite{KwPa}.
In Theorem~\ref{thabove}, the constant $\widehat\nu_0>0$ is related to the Montgomery model \cite{Mon}  whose spectral analysis has a long story including recently (see \cite{HeLe} and references therein). For $\rho\in\R$, we introduce, in $L^2(\R)$,  the operator 
\[S(\rho)=D_r^2+(r^2-\rho)^2\,, \]
and denote its lowest eigenvalue by $\mu^{\rm Mon}(\rho)$. Then
\begin{equation}\label{eq:Mont}
\widehat\nu_0:=\inf_{\rho\in\R}\mu^{\rm Mon}(\rho)=\mu^{\rm Mon}(\rho_0)\,,
\end{equation}
where $\rho_0\in\R$ is the  unique minimum of $\mu^{\rm Mon}$,  which has been later shown to be non degenerate \cite{Qmath10}. 
Finally,  the normalized positive  ground state $\psi_0\in L^2(\R)$ of $S(\rho_0)$  belongs to  the Schwartz space $S(\mathbb R)$ and is an even function. 
\section{Model operator for non-uniform magnetic fields}\label{sec:Nmod}

Given real parameters $\eta,\zeta,\gamma$ and $\theta$, we  consider the operator 
\begin{equation}\label{n1a}
\begin{array}{ll}
P^{h,\eta,\zeta}_{0;\gamma,\theta} &:= (h D_r   - \sin \theta \,t -  \cos \theta \, (\eta s +
 \zeta r) t  )^2\\&\quad
 + (h D_s + \cos \theta  \,t  -  \sin \theta \,  (\eta s +\zeta r)\, t
 +  \gamma \frac{r^2}{2})^2 \\&\quad  + h^2 D_t^2\;,
\end{array}
\end{equation}
on $\R^2 \times \R^+$ (actually in a neighborhood of $(0,0,0)$).
 Let  us fix a positive constant $M$.  We assume  that 
\begin{equation}\label{eq:eta..<M}
\eta,\zeta,\gamma\in[-M,M]\,.
\end{equation}

We note, when $\eta=\zeta=0$, we recover the model studied in \cite[Sec.~11]{HelMo4}. 
Our aim is to compare this situation with  that  when $\eta=\zeta=0$. 
Our main result on this model is Proposition~\ref{lemestmodzone2bis} below, which is useful in our derivation of the lower bound matching with the asymptotics in Theorem~\ref{thm:main}. The lower bound in this proposition  is uniform with  respect to the various parameters  appearing in \eqref{n1a} provided \eqref{eq:eta..<M} holds and   $h$ is sufficiently small.

Let us look at this model more carefully.
We proceed essentially like in the case $\eta=\zeta=0$.
We do the following scaling 
\begin{equation}\label{newscaling}
r=h^\frac 13 \hat r\;,\; s=h^\frac 13 \hat s\;,\; t=h^\frac 12 \hat
t\;.
\end{equation}
After division by $h$, this leads to  (forgetting the hats)
\begin{equation}\label{n2a}
\begin{array}{l}
P_{1;\gamma,\theta}^{h,\eta,\zeta}:= \left(h^\frac 16  D_r   - \sin \theta \, t - h^{\frac13}  \cos \theta\, t   (\eta s + \zeta r) \right)^2\\
\quad \quad 
 + \left(h^\frac 16 D_s  + \cos \theta \, t + 
 h^\frac 16 \gamma \frac{r^2}{2} - h^\frac 13    \sin \theta\,t (\eta
 s + \zeta  r)\right)^2 +
 D_{ t}  ^2
\end{array}
\end{equation}
on $\R \times \R \times \R^+$. \\ 
 Hence we have 
\[\sigma(P_{0;\gamma,\theta}^{h,\eta,\zeta})=h\,\sigma(P_{1;\gamma,\theta}^{h,\eta,\zeta}).\]
Unlike the case where $\eta=\zeta=0$, we can no more perform a partial Fourier transform in the $s$-variable.
But we  can  
rewrite  this operator as in the following lemma.
\begin{lemma}\label{lem:n4a}
It holds,
\begin{equation}\label{n4a}
P_{1;\gamma,\theta}^{h,\eta,\zeta} = D_t^2 + (t -  h^\frac 16    L_{1,\gamma,\theta}) ^2 +
h^\frac 13  ( L_{2,\gamma,\theta} ^{h,\eta,\zeta})^2\;,
\end{equation}
where
\begin{equation} \label{eq:L1,L2}  
\begin{array}{ll}
  L_{1;\gamma,\theta}   & =  \sin \theta D_r - \cos \theta \left(\frac \gamma
2 r^2+ D_s \right) \;, \\
  L^{h,\eta,\zeta}_{2;\gamma,\theta}   & :=  
\cos \theta D_r + \sin \theta \left(\frac \gamma 2 r^2 + D_s\right)
 -   h^\frac 16  ( \zeta r + \eta s)  t \,.
\end{array}
\end{equation}
\end{lemma}
Note that to compare with the case considered in \cite{HelMo4} ($\eta=\zeta=0$) we can write
\begin{equation}\label{eq.L2=L20}
 L_{2;\gamma,\theta}^{h,\eta,\zeta}  = L_{2;\gamma,\theta} -   h^\frac 16  ( \zeta r + \eta s)  t\,,
\end{equation}
where  $L_{2;\gamma,\theta}:=L_{2;\gamma,\theta}^{0,0,0}$.
\begin{proof}[Proof of Lemma~\ref{lem:n4a}]
 Let  $P_{1;\gamma,\theta}^h :=P_{1;\gamma,\theta}^{h,0,0}$. Then
 (see \cite[Eq.~(11.4)]{HelMo4})
 \[ P_{1;\gamma,\theta}^h= D_t^2 + (t -  h^\frac 16    L_{1,\gamma,\theta}) ^2 +
h^\frac 13  ( L_{2,\gamma,\theta} )^2\,.\]
With  $p=(\eta s + \zeta r) t$, we have 
\[
 P_{1;\gamma,\theta}^{h,\eta,\zeta}=  P_{1;\gamma,\theta}^h
+h^{\frac13}\big[ -2(h^{\frac16} p) L_{2;\gamma,\theta}   - h^{\frac16}\big(\cos\theta (D_r p) +\sin\theta (D_sp)\big)+(h^{\frac16}p)^2  \big]\,. \]
Finally, we observe by \eqref{eq.L2=L20},
\[(L_{2;\gamma,\theta}^{h,\eta,\zeta} )^2=( L_{2,\gamma,\theta} )^2-2(h^{\frac16} p) L_{2;\gamma,\theta}   - h^{\frac16}\big(\cos\theta (D_r p) +\sin\theta (D_sp)\big)+(h^{\frac16}p)^2\,. \]
\end{proof}
When $\eta=\zeta =0$, this is the operator studied in \cite{HelMo4}, modulo a  Fourier transformation  with respect to the $s$ variable. Let us recall the following important result \cite[Lem.~13.4]{HelMo4} corresponding to the case $(\eta,\zeta)=(0,0)$.
\begin{proposition}[Helffer-Morame]\label{lemestmodzone2}\
For any $C_0>0$, $\delta  \in ]0, \frac{1}{3}[$ and $M>0$, there exist positive constants $C$ and $h_0$ such
that, for all $\theta\in\R$, $|\gamma|\leq M$, and $h\in ]0,h_0]$,  we have,  for any $u\in C_0^\infty (]-C_0h^{\delta}, C_0h^{\delta }[\times \R\times 
\overline{\R_+})$
\begin{equation}\label{estglob2} 
\langle P_{0;\gamma,\theta}^{h,0,0} u\, ,\,  u \rangle 
\geq \big(h\Theta_0+h^{\frac 43}c^{\rm conj}(\gamma, \theta) -
 C (h^{\frac{11}{8}} + h^{\delta + \frac{13}{12}})\big)\; \| u\| ^2 \; ,
\end{equation} 
where
\[c^{\rm conj}(\gamma, \theta):= \Big(\frac 12\Big)^\frac 23 \delta_0^\frac 13 
 |\gamma|^\frac 2 3 ( \delta_0 \sin^2 \theta + \cos^2 \theta)^\frac 13 \, 
{\hat \nu_0}, \]
and $P^{h,0,0}_{0;\gamma,\theta}$ is the operator introduced in (\ref{n1a}). 
\end{proposition} 
\begin{remark}\label{rem:op-P2}
The underlying estimate in Proposition~\ref{lemestmodzone2} is in  fact
\[ \langle P_{1;\gamma,\theta}^{h,0,0} u\, ,\,  u \rangle 
\geq \big(\Theta_0+h^{\frac 13}c^{\rm conj}(\gamma, \theta) -
 C (h^{\frac{3}{8}} + h^{\delta + \frac{1}{12}})\big)\; \| u\| ^2.\]
 \end{remark}
 We can not directly compare $ P_{1;\gamma,\theta}^{h,\eta,\zeta} $ and $P_{1;\gamma,\theta}^{h,0,0}$ but this can be done by introducing   
  a small perturbation of  $P_{1;\gamma,\theta}^{h,0,0}$ whose spectrum is just lifted. To achieve this goal we introduce for $\tau >0$
  \[P_{1;\gamma,\theta,\tau}^{h}:=D_t^2 + (t  -  h^\frac 16    L_{1,\gamma,\theta} )^2 +
(1-h^\tau) h^\frac 13  ( L_{2,\gamma,\theta})^2,\]
where we have modified the coefficient of $( L_{2,\gamma,\theta})^2$ by $\epsilon= h^{1/3 +\tau}$. Heuristically this leads to a maximal shift of the bottom of the spectrum
by $\mathcal O (h^{1/3 +\tau})$. More precisely, we show by a slight variation of the argument in \cite[Lem.~13.3]{HelMo4}
\begin{proposition}
For all $\tau\in\,]0,1[$, for any $C_0>0$, $\delta  \in ]0, \frac{1}{3}[$ and $M>0$, there exist positive constants $C$ and $h_0$ such
that, for all $\theta\in\R$, $|\gamma|\leq M$, and $h\in ]0,h_0]$,  we have,  for any $u\in C_0^\infty (]-C_0h^{\delta}, C_0h^{\delta }[\times \R\times 
\overline{\R_+})$
\begin{equation}\label{eq:var}
 \langle  P_{1,\gamma,\theta, \tau}^{h}  u\, ,\,  u \rangle 
\geq \big(\Theta_0+h^{\frac 13}c^{\rm conj}(\gamma, \theta) -
 C (h^{\tau+\frac13}+h^{\frac{3}{8}} + h^{\delta + \frac{1}{12}})\big)\; \| u\| ^2\,.
 \end{equation}
 \end{proposition}
Note that the estimate in Proposition~\ref{lemestmodzone2} holds without  constraint on the
support of the function in $s$. This will not be the  case for $(\eta,\zeta)\not=0$.\\
We now compare $ \langle P_{1;\gamma,\theta}^{h,\eta,\zeta} u,u\rangle$ and   $\langle  P_{1;\gamma,\theta,\tau}^{h} u,u\rangle$  when 
\[ u\in C_0^\infty\big(\,]-C_0h^{\delta -\frac13}, C_0h^{\delta-\frac13 }[\times\,]-C_0h^{\delta-\frac13 }, C_0h^{\delta-\frac13 }[ \times \overline{\R_+}\,\big).\]
and $\eta,\zeta$ satisfies \eqref{eq:eta..<M}.

Let us fix 
\begin{equation}\label{eq:assdeltatau}
\delta\in\,]\frac14,\frac13[ \mbox{  and  } \tau\in]0,\frac1{6}[\,.
\end{equation}
The estimates below  hold uniformly  with respect to $u$, $\theta\in\R$ and $\eta,\zeta,\gamma$  satisfying \eqref{eq:eta..<M}.\\
Comparing $L_{2,\gamma,\theta}^{h,\eta,\zeta}$ and $L_{2,\gamma,\theta}$ in \eqref{eq.L2=L20}, we find\footnote{We use $2ab\leq \varepsilon a^2+\varepsilon^{-1}b^2$ with $\varepsilon=h^\tau$,  $a=\|L_{2,\gamma,\theta}^{h,\eta,\zeta}u\|$  and $b=\|L_{2,\gamma,\theta}u\|^2$.}, for all $\tau>0$,
\[ \langle (L_{2,\gamma,\theta}^{h,\eta,\zeta})^2u,u\rangle =\|L_{2,\gamma,\theta}^{h,\eta,\zeta}u\|^2\geq (1-h^\tau)\|L_{2,\gamma,\theta}u\|^2+(1-h^{-\tau})\|(L_{2,\gamma,\theta}^{h,\eta,\zeta}-L_{2,\gamma,\theta})u\|^2\,.\]
Consequently,
\begin{multline}\label{n4aa0}
\langle  P_{1;\gamma,\theta}^{h,\eta,\zeta} u,u\rangle \geq \langle \big(D_t^2 + (t  -  h^\frac 16    L_{1;\gamma,\theta} )^2 +
(1-h^\tau)h^{\frac 13}  ( L_{2;\gamma,\theta})^2\big)u,u\rangle\\ - h^{\frac 13-\tau} 
||( L_{2;\gamma,\theta}^{h,\eta,\zeta}-L_{2;\gamma,\theta} )u||^2 \;.
\end{multline}
This implies  (see \eqref{eq.L2=L20} and  the condition on the support of $u$),
\begin{equation}\label{n4aa00} 
\langle  P_{1;\gamma,\theta}^{h,\eta,\zeta} u,u\rangle \geq  \langle P_{1;\gamma,\theta,\tau}^h  u,u\rangle   - C (\eta^2+\zeta^2)h^{2\delta-\tau}
||tu||^2 \;,
\end{equation}
  where we used  (see \eqref{eq.L2=L20}) 
$$  L_{2;\gamma,\theta}^{h,\eta,\zeta} - L_{2;\gamma,\theta} =h^{1/6} t\, \mathcal O( (|s|+|r|))=t \, \mathcal O(h^{\delta-\frac16})$$  in the support of $u$.\\
By  \eqref{eq:var} and \eqref{n4aa00} we have
\begin{multline}\label{n4aa1}
\langle   P^{h,\eta,\zeta}_{1;\gamma,\theta} u,u\rangle   \geq   \big(\Theta_0 +c^{\rm conj}(\gamma, \theta)h^{1/3} 
- C (h^{\frac{3}{8}} +h^{\delta+\frac1{12}}+h^\tau)\big) \| u\| ^2\\
  -  C(\eta^2+\zeta^2)  h^{2\delta-\tau}\|tu\|^2 \;.
\end{multline}
 Note that by \eqref{eq:assdeltatau} we have
\[ h^{\frac{3}{8}} +h^{\delta+\frac1{12}}+ h^{\tau+\frac13} +h^{2\delta-\tau}=\mathcal O(h^{\frac13+\varsigma})\,,\]
for some $\varsigma=\varsigma(\delta,\tau)>0$. \\
Consequently,  there exist $C$, $\varsigma >0$ and $h_0$ such that, $\forall h\in ]0,h_0]$,
\begin{equation}\label{n4aa2}
\langle   P_{1;\gamma,\theta,\tau}^{h,\eta,\zeta} u,u\rangle\geq \big(\Theta_0 +c^{\rm conj}(\gamma, \theta)h^{1/3} 
- C h^{\frac13+\varsigma}\big) \| u\| ^2-  C h^{\frac13+\varsigma}\|tu\|^2\,,
\end{equation}
for any $u\in C_0^\infty (\,]-Ch^{\delta -\frac 13}, Ch^{\delta -\frac
  13}[^2\times\overline{\R_+}\,)$. \\
  By coming back to the initial coordinates, we get the following generalization  of Proposition~\ref{lemestmodzone2}.
\begin{proposition}\label{lemestmodzone2bis} \
Let $C_0,M>0$ and $\delta  \in ]\frac 14, \frac{1}{3}[$ be given. 
    There exist  positive constants $C$, $h_0$, and $\varsigma$, such
that, 
for all $h\in ]0,h_0]$, $\theta\in\R$ and $\gamma, \eta,\zeta \in[-M,M]$,  
  we have, for any  $u\in C_0^\infty (]-C_0h^{\delta }, C_0h^{\delta }[^{\,2} \, \times \,
\overline{\R^+})$,
\begin{equation}\label{estglob2bis} 
\langle P_{0,\gamma,\theta}^{h,\eta,\zeta} u\, ,\,  u \rangle 
\geq \big(h\Theta_0+h^{\frac 43} c^{\rm conj}(\gamma, \theta) -
 C h^{\frac 43 +\varsigma}\big)\| u\| ^2- C h^{\frac13+\varsigma}\|tu\|^2 . 
\end{equation} 
\end{proposition} 
Note here that the last term will be small when considering localized states satisfying \eqref{eq:dec-gs1}.
\section{Localization of bound states}\label{sec:loc}

We  recall that  the bound states of the operator $P^h_\Ab$ in \eqref{eq:op} are localized on the boundary  near the curve   where the magnetic field is tangent  to the boundary $\partial\Omega$.  The localization  is related with 
the analysis of a family of   model operators in the  half-space \cite{LuPa5}.

Consider $\R_+^3:=\{(x_1,x_2,x_3)\in\R^3\,|\,x_1>0\}$ and  the Neumann realization in $\R_+^3$ of the operator,
\[H(\nu)=D_{x_1}^2+D_{x_2}^2+(D_{x_3}+x_1\cos\nu-x_2\sin\nu)^2\,, \]
where $\nu\in[-\frac\pi2,\frac\pi2]$. \\More precisely, $H(\nu)$ is self-adjoint in    $L^2(\mathbb R_+^3)$ with the following domain
\[{\rm Dom}\big(H(\nu)\big)=\{u\in L^2(\mathcal \R^3_+)\,|\,H(\nu)u \in L^2(\mathcal \R^3_+),~
\partial_{x_1}u|_{x_1=0}=0 \}\,. \]
   We denote by 
\begin{equation}\label{eq:gs-hp}
\sigma(\nu)=\inf_{\nu\in[-\frac\pi2,\frac\pi2]}{\rm  spec}\,\big(H(\nu)\big)\,.
\end{equation}

We gather some properties of the  lowest eigenvalue $\sigma(\nu)$ (see \cite{LuPa5}, \cite{HelMo2}, and  \cite[Sec.~3.3]{HelMo4}):
\begin{proposition}\label{prop:gs-hp}~
The following properties hold for the lowest eigenvalue $\sigma(\nu)$ of $H(\nu)$:
\begin{itemize}
\item For all $\nu\in[-\frac\pi2,\frac\pi2]$, $\sigma(-\nu)=\sigma(\nu)$.
\item $[0,\frac\pi2]\ni\nu\mapsto\sigma(\nu)$ is monotone increasing and $\sigma(0)=\Theta_0$.
\item $\sigma(\nu)\geq\Theta_0\cos^2\nu+\sin^2\nu$.
\item As $\nu\to0$, $\sigma(\nu)=\Theta_0+\sqrt{\delta_0}\,|\nu|+\mathcal O(\nu^2)$.
\end{itemize}
Here we recall that $\Theta_0$ and $\delta_0$ are introduced in \eqref{eq:Th0}.
\end{proposition}
Let us return to the magnetic field in \eqref{eq:B}. Recall that, for $x\in\Omega$, $p(x)\in\partial\Omega$ satisfies  ${\rm dist}(x,\partial\Omega)={\rm dist}(x,p(x))$, and it is uniquely defined when $x$ is sufficiently  close to the boundary. For all $x\in\overline\Omega$, we introduce $\nu(x)\in[-\frac\pi2,\frac\pi2]$  by
\begin{equation}\label{eq:def-nu}
(\Bb\cdot\Nb)(p(x))=\sin\nu(x)\,. 
\end{equation}
Hence $\nu(x)=0$ implies that $\Bb (p(x))$ is tangent to $\partial \Omega$ at $p(x)$, in other words that $x$ belongs to $\Gamma$ (see \eqref{eq:1.4}).
Now we recall the following lower bound  related to the operator $P_\Ab^h$ established in \cite[Thm.~4.3]{HelMo4}:
\begin{proposition}\label{prop:lb-qf}
Under Assumption \eqref{eq:B}, there exist constants $C,h_0>0$ such that, for all $h\in(0,h_0]$ and $u\in H^1(\Omega)$, we have
\[ \int_\Omega|(h\nabla-i\Ab)u|^2dx\geq \int_\Omega (hW_h(x)-Ch^{5/4})|u(x)|^2dx\,,\]
where
\[ W_h(x)=\begin{cases}
1&{\rm if~}{\rm dist}(x,\partial\Omega)\geq 2h^{3/8}\\
\,\sigma(\nu(x))&{\rm if~}{\rm dist}(x,\partial\Omega)\leq 2h^{3/8}
\end{cases}\,.\]
\end{proposition}
 If additionally $u\in H^1_0(\Omega)$, we have for some positive constant $C_0$  the stronger lower bound
\[ \int_\Omega|(h\nabla-i\Ab)u|^2dx\geq (h-C_0h^{5/4}) \int_\Omega |u|^2\,dx\,.\]
 Combining the lower bound in  Proposition~\ref{prop:lb-qf} with the following leading term expansion of the lowest eigenvalue (see \cite[Thm.~4.4]{HelMo4})
\begin{equation}\label{eq:abc}
\lambda_1^N(\Ab,h)=\Theta_0h+o(h)\,,
\end{equation}
we get decay estimates for the  ground states.
Let us recall these localization estimates (see \cite[Sec.~9.4]{FoHe2} for details). 
\begin{proposition}\label{prop:dec-gs}~\\
Given  $M>0$, there exists a  positive constant $\alpha$ such that,  if $u_h$ is a normalized bound state  of $\mathcal P_h$    with eigenvalue $\lambda(h)\leq Mh$, then as $h\to0_+$,
\begin{equation}\label{eq:dec-prop}
\int_\Omega\Big(|u_h(x)|^2+h^{-1}|(h\nabla-i\Ab)u_h|^2 \Big)\exp\Big(\frac{\alpha \,{\rm dist}(x,\partial\Omega)}{h^{1/2}}\Big)dx=\mathcal O(1)\,.
\end{equation}
 Furthermore, there exist  constants $\alpha_1,\epsilon_0>0$ such that, as $h\to0_+$,
\[\int_{\{{\rm dist}(x,\partial\Omega)<\epsilon_0\}} \Big(|u_h(x)|^2+h^{-1}|(h\nabla-i\Ab)u_h|^2 \Big)\exp\Big(\frac{\alpha_1 \,d_\Gamma (x)}{h^{1/4}}\Big)dx=\mathcal O(1)\,, \]
where
\begin{equation}\label{eq:d-Gam}
d_\Gamma(x)={\rm dist}_{\partial\Omega}(p(x),\Gamma)\,,
\end{equation}
and ${\rm  dist}_{\partial\Omega}$ is the geodesic distance on $\partial\Omega$.
\end{proposition}
Hence we have two levels of localization, first a strong one near $\partial \Omega$ and then an additional  but  weaker one near $\Gamma$.  Along the proof of  Theorem~\ref{thm:main}, we will only use \eqref{eq:dec-prop} and  generalizations/consequences of it, as explained in  the below remark.
\begin{remark}[Applications of Proposition~\ref{prop:dec-gs}]\label{rem:genDecE}~\\ 
Let $u_h$ be   a normalized ground state of $P_\Ab^h$.
\begin{enumerate}
\item  By \eqref{eq:abc},   the hypothesis in Proposition~\ref{prop:dec-gs} holds, hence the ground state $u_h$ satisfies \eqref{eq:dec-prop} and \eqref{eq:d-Gam}.
\item
Pick an arbitrary point $x_0\in\Gamma$. In the coordinates introduced in  \eqref{eq:label(r,s,t)}, where $t(x)={\rm dist}(x,\partial\Omega)$, $r(x)=d_\Gamma(x)$ 
 and $u_h(x)=\tilde u_h(r,s,t)$ (see \eqref{eq:tilde-u}), we deduce from \eqref{eq:dec-prop} the  following weaker, but quite useful estimates.
For any $n\geq 0$, 
\begin{equation}\label{eq:dec-gs1}
\int_{\tilde V_{0}} t^n|\tilde u_h|^2dsdrdt=\mathcal  O(h^{n/2} ),
\end{equation}
and
\begin{equation}\label{eq:dec-gs2}
\int_{\tilde V_{0}} t^n |(h\nabla_{r,s,t}-\tilde\Ab)\tilde u_h|^2drdsdt=\mathcal  O(h^{ 1+\frac n2} )\,,\\  
\end{equation}
where $\tilde V_{0}:=\tilde V_{x_0}$ and $\tilde\Ab$ are introduced in \eqref{eq:tilde-V-x0}  and  \eqref{eq:tA} respectively.
\end{enumerate}
\end{remark}

\section{Estimating the quadratic form}\label{sec:qf}

\subsection{A comparison estimate}
 We  fix  $\delta$ and $\epsilon_2$ satisfying 
\begin{equation}\label{eq:cond-delta}
\frac{5}{18}<\delta<\frac13 \mbox{ and }  0 < \epsilon_2 < 1 \,.
\end{equation}
We also fix  $R_0>0$,   $h_0>0$, $x_0\in\Gamma$ and introduce for $h\in (0,h_0]$ the  set
\begin{multline}\label{eq:Q(x0)}
Q_h(x_0, R_0,\delta,\epsilon_2)\\
=\{x\in\Omega\,:\,|r(x)-r_0|\leq R_0h^\delta,~|s(x)-s_0|\leq R_0h^\delta,~0<t(x)<\epsilon_2\}
\end{multline}
where  $(r(x),s(x),t(x))$ are introduced in \eqref{eq:Phi-x0} and,  since $ x_0\in\Gamma$,  
\begin{equation}\label{eq.def.y0}
 y^{0}:=(r_0,s_0,t_0):=(r(x_0), s(x_0),t(x_0))=(0,s(x_0),0)\,.
\end{equation}
For simplicity, we omit most of the time the reference to $\delta$ and  $\epsilon_2$.

Let $\tilde\Ab=(\tilde A_1^{(2)},\tilde A_2^{(2)},\tilde A_3^{(2)})$ be the magnetic potential associated with  $\Ab$ via \eqref{eq:tA}, with  $y=(y_1,y_2,y_3)=(r,s,t)$ (see \eqref{eq:label(r,s,t)}). We introduce the following magnetic potential
\begin{equation}\label{eq:tilde-A-2}
\tilde\Ab^{(2)}(y) =\sum_{|\beta|\leq 2}\frac{\partial^\beta\tilde\Ab}{\partial y^\beta}(y^0)\frac{(y-y^0)^\beta}{\beta!},
\end{equation}
which is the quadratic Taylor expansion of $\tilde\Ab$ at $ y^0$.
We introduce the quadratic form  associated with   the magnetic potential $\tilde\Ab^{(2)}$ as follows
\begin{multline*}
q_{\tilde\Ab^{(2)}}^h(u)=\int_{\tilde Q_h(x_0,R_0)}(1-r\kappa_g(x_0))\Big(|(hD_t-\tilde A_3^{(2)})u|^2\\
+(1+2r\kappa_g(x_0))|(hD_s-\tilde A_2^{(2)})u|^2+|(hD_r-\tilde A_1^{(2)})u|^2\Big)drdsdt\,, 
\end{multline*}
where 
\begin{equation}\label{eq:tilde-Q}
\tilde Q_h(x_0,R_0,\delta,\epsilon_2)=\{(r,s,t)~:~\max(|r|,|s-s_0|)<R_0h^\delta,~0< t<\epsilon_2\}\,,
\end{equation}
and (see \eqref{eq:kg-kn})
\[\kappa_g(x_0) \mbox{   is the geodesic curvature of } \Gamma \mbox{ at }
x_0\,.\]
The next lemma compares the quadratic forms $u\mapsto q_{\tilde\Ab^{(2)}}^h(u)$ and  $u\mapsto q_\Ab^h(u)$ introduced in \eqref{eq:qf-y}.  The errors that will arise are controlled by the  following energy
\begin{equation}\label{eq:Mh(u)}
 M_h(u)=\sum_{n=0}^6h^{-n/2}\int_\Omega t(x)^n\Big(|u|^2+h^{-1}|(h\nabla-i\Ab)u|^2 \Big)dx\,,\end{equation}
where $t(x)={\rm  dist}(x,\partial\Omega)$.  Notice that,
\begin{equation}\label{eq:cond-u}
\begin{aligned}
{\rm (a)}\quad \int_\Omega|u|^2dx&\leq M_h(u)\,,\\
{\rm (b)} \quad \int_\Omega |(h\nabla-i\Ab)u|^2dx&\leq M_h(u)h\,,\\
{\rm (c)} \quad \int_\Omega t(x)^n \big(|u|^2+h^{-1}|(h\nabla-i\Ab)u|^2\big)dx&\leq  M_h(u)h^{n/2}\quad(1\leq n\leq 6)\,.
\end{aligned}
\end{equation}  
\begin{lemma}\label{lem:10.1}
There exist constants $C,h_0,\varsigma_0>0$ such that, for all $h\in(0,h_0]$ and $u\in H^1(\Omega)$ satisfying 
  $ {\rm supp\,}\,u\subset Q_h(x_0,R_0)$, we have
\begin{equation}\label{eq:aa7.6}
(1-C\, h^{2\delta})q_{\tilde\Ab^{(2)}}^h(u)-CM_h(u)h^{\frac43+\varsigma_0} \leq q_\Ab^h(u)\leq (1+C\, h^{2\delta})q_{\tilde\Ab^{(2)}}^h(u)+CM_h(u)h^{\frac43+\varsigma_0}\,.
\end{equation}
\end{lemma}
\begin{proof}
 Let us recall  two useful estimates whose proof  does not require that the magnetic field $\curl\Ab$ is constant (see \cite[Lem.~10.1]{HelMo4}):
\begin{multline} \label{eq:a7.6}
 q_\Ab^h(u)\geq (1-Ch^{2\delta})q_{\tilde\Ab^{(2)}}^h(u)-C\|t^{1/2}(h D_x-\Ab)u\|^2 \\-C\big(q_{\tilde\Ab^{(2)}}^h(u)\big)^{1/2}\| (h^{3\delta}+h^{2\delta}t+h^\delta t^2+t^3)u\|\\
 -C\| (h^{3\delta}+h^{2\delta}t+h^\delta t^2+t^3)u\|^2\,,
 \end{multline}
 and
\begin{multline}\label{eq:a7.7}
 q_\Ab^h(u)\leq (1+Ch^{2\delta})q_{\tilde\Ab^{(2)}}^h(u)+C\|t^{1/2}(h D_x-\Ab)u\|^2\\
 +C\big(q_{\tilde\Ab^{(2)}}^h(u)\big)^{1/2}\| (h^{3\delta}+h^{2\delta}t+h^\delta t^2+t^3)u\|\\
 +C\| (h^{3\delta}+h^{2\delta}t+h^\delta t^2+t^3)u\|^2\,.
\end{multline}
In the  sequel we use the notation $\mathcal O(c_h h^{\rho+})$ in the following manner
\begin{equation}\label{eq:O-new}
 f_h=\mathcal O(c_h h^{\rho+}) \mbox{ if and only if } \exists\,\epsilon>0 \mbox{ s.t. }f_h=\mathcal O(c_hh^{\rho+\epsilon}).
\end{equation} 
Since we have assumed \eqref{eq:cond-delta}, we have 
$$\min(6\delta\,,\, 2\delta+1\,,\,3\delta+\frac12\,,\,2-2\delta )>\frac43\,.$$
 We can now estimate  the error terms appearing  in \eqref{eq:a7.6} and \eqref{eq:a7.7}. We deduce from \eqref{eq:cond-u}  (a) that
\[\|h^{3\delta} u\|^2=\mathcal O(M_h h^{6\delta})=\mathcal O(M_hh^{\frac43+})\,, \]
  where we write $M_h$ instead of $M_h(u)$ for the sake of simplicity.

Using again \eqref{eq:cond-u} with $n=1$, $n=2$, $n=4$ and $n=6$, we get
\begin{equation*}
\begin{array}{rl}
\|t^{1/2}(hD_x-\Ab)u\|^2 & =  \mathcal O (M_hh^{\frac54}),\\ \|h^{2\delta} tu\|^2 & = \mathcal O (M_hh^{ 4\delta  + 1}) \,,\\
\|h^{\delta} t^2u\|^2& =\mathcal O(M_hh^{2\delta+2} )\,,\\
 \|t^3u\|^2 & = \mathcal O(M_h h^{3 })\,.\end{array}
\end{equation*}
 Consequently,
 \begin{equation}\label{eq:remainder-lb}
 \|t^{1/2}(hD_x-\Ab)u\|^2+\| (h^{3\delta}+h^{2\delta}t+h^\delta t^2+t^3)u\|^2=\mathcal O (M_h h^{\frac43+})\,. 
 \end{equation}
Notice that $|\tilde\Ab-\tilde\Ab^{(2)}|=\mathcal O(h^{3\delta})+\mathcal  O(t^3)$ in $\tilde Q_h(x_0,R_0)$. By the triangle inequality  and  \eqref{eq:qf-y} 
\[ q_{\tilde\Ab^{(2)}}^h(u)\leq C\, \big( q_{\Ab}^h(u)+\|t^3u\|^2\big)\,.\]
So by  using  \eqref{eq:cond-u} we get
\[ q_{\tilde\Ab^{(2)}}^h(u)=\mathcal  O(M_hh)\,.\]
Consequently, the foregoing estimate and \eqref{eq:remainder-lb}  yield,
\[ \big(q_{\tilde\Ab^{(2)}}^h(u)\big)^{1/2}\| (h^{3\delta}+h^{2\delta}t+h^\delta t^2+t^3)u\|=\mathcal  O(M_hh)\,.\]
This finishes the proof of \eqref{eq:aa7.6}.
\end{proof}

\subsection{Normal form}

Recall that we have fixed an arbitrary point $x_0\in\Gamma$ and denoted its coordinates, in the $(r,s,t)$-frame,   by $(0,s_0,0)$.
Let us also recall  that the  magnetic field $\Bb(x_0)$ can be expressed by  \eqref{eq:B-r,t=0}. 

Performing an  appropriate gauge transformation  on the set $\tilde Q_h(x_0,R_0)$ introduced in \eqref{eq:tilde-Q}, will yield  a convenient normal form  of  the magnetic potential $\tilde{\Ab}^{(2)}$ introduced in \eqref{eq:tilde-A-2}. 

\begin{lemma}\label{lem:normal-form}
There exist positive constants $C$ and $\widehat C$, and for all $x_0\in\Gamma$, there exist $\check\kappa,\zeta\in [-\widehat C,\widehat C] $ and a smooth function  $\check p$ on a neighborhood of $\tilde Q_h(x_0,R_0,\delta,\epsilon_2)$, such that,
\[ |\tilde\Ab^{(2)}(r,s,t)-\Ab^{00}(r,s,t)+\nabla \check p(r,s,t)|\leq C\, \big(r^3+t^2+|s-s_0|^3\big)\,,\]
where
\[\Ab^{00}(r,s,t)=\Big(ta_1(r,s),ta_2(r,s)+\frac12\kappa_{n,\Bb}(x_0)r^2,0 \Big),\]
$\kappa_{n,\Bb}(x_0)$ is introduced in \eqref{eq:k-n,B}, and
\begin{align*}
a_1(r,s)&=  \sin\theta(s_0) +\big(\zeta r+\check\kappa (s-s_0)\big)\cos\theta(s_0)\,,\\
a_2(r,s)&=-\cos\theta(s_0)+r\kappa_g(x_0)\cos\theta(s_0)+ (\zeta r+\check\kappa (s-s_0) )\sin\theta(s_0)\,.
\end{align*} 
 Here $\theta(s_0)$ is the angle introduced in   \eqref{eq:theta-r,t=0} with $x=x_0$.
\end{lemma}
This lemma is an extension of Lemma 9.1 in \cite{HelMo4} to the case when the magnetic field is not necessarily constant. In the constant magnetic field case we have $\zeta=0$ and
$\check \kappa =\kappa_g(x_0)$, where $\kappa_g$ is  the  geodesic curvature introduced in \eqref{eq:kg-kn}.
Note that we do not try at the moment to explicitly compute $\check \kappa$ and $\zeta$  in
the general case. We plan indeed to show that the   result on  the lowest eigenvalue is
independent
 of $\check \kappa$ and $\zeta$.
 
\begin{proof}[Proof of Lemma~\ref{lem:normal-form}]
Our goal is to determine the Taylor expansion up to
 order~$1$ of the magnetic field vector and corresponding magnetic
 field $2$-form in the variables $(r,s,t)$, the Taylor expansion being
 computed at $t=r=0$ and $s=s_0$. Up  to a  translation, we assume that  $s_0=0$.\\
Writing the magnetic vector field  in \eqref{eq:A} as
\begin{equation}
\Bb = \widetilde b_1 \pa_r + \widetilde b_2 \pa_s + \widetilde b_3
\pa_t\;,
\end{equation}
the Taylor expansion of order $1$ at $(0,0,0)$ takes the form
\begin{equation}
\begin{array}{ll}
\widetilde b_1(r,s,t) &=  \cos \theta + \gamma_1 r + \delta_1 s +
\sigma_1 t+\mathcal O(r^2+s^2+t^2)\;,\\
\widetilde b_2(r,s,t) & =\sin \theta + \gamma_2r + \delta_2  s + \sigma_2 t+\mathcal O(r^2+s^2+t^2)\;,\\
\widetilde b_3 (r,s,t)& = \gamma_3 r +  \sigma_3 t+\mathcal O(r^2+s^2+t^2)\;.
\end{array}
\end{equation}
 where $\theta=\theta(s_0)$ and where we used \eqref{eq:kappa-n-B}-\eqref{eq:B-r,t=0}.
Here we have used that by definition of the coordinate $r$, the
function $(r,s)\mapsto \widetilde b_3(r,s,0)$ vanishes exactly at
order $1$ on $r=0$. Note that $\gamma_3$ is $\kappa_{n,\Bb}(x_0)$, introduced in  \eqref{eq:k-n,B}.\\

We now express that on $t=0$ the norm of $\Bb$ should be one.  In fact 
\begin{equation}\label{eq:normB}
  |\Bb|^2=\sum\limits_{1\leq i,j\leq 1}g_{ij}\tilde b_i\tilde  b_j+\tilde b_3^2
\end{equation}
where the coefficients $g_{ij}$ can be computed by  \eqref{eq:g0-y},  \eqref{eq:g-ij} and \eqref{eq:alpha(r,s)}.\\
 For $t=0$, this reads
\begin{equation}
(\widetilde b_1 (r,s,0)^2 + \alpha (r,s) (\widetilde b_2(r,s,0))^2
 + (\widetilde b_3(r,s,0))^2 =1\;,
\end{equation}
where $\alpha(r,s)$  is introduced in \eqref{eq:alpha(r,s)} and satisfies \eqref{eq:alpha(r,s)*}. 
We expand the last formula around $t=r=s=0$.   This leads, by taking
$t=0$
 and considering the coefficients of $r$ and $s$,  to the two
 identities
\begin{equation*}
\gamma_1 \cos \theta + \gamma_2 \sin \theta -   \kappa_g(x_0) \sin^2
\theta =0\;,
\end{equation*}
and
\begin{equation*}
\delta_1 \cos \theta + \delta_2 \sin \theta =0\;.
\end{equation*}
So it is natural to introduce  the new parameters $\hat\kappa$   and $\zeta$  as follows
\begin{equation}\label{eq:delta1,2}
\check \kappa = - \delta_1 \sin \theta + \delta_2 \cos \theta,\quad 
\zeta=- \gamma_1 \sin \theta + (\gamma_2-\kappa_g(x_0)\sin\theta) \cos \theta\,.
\end{equation}
 So we observe that
\[ \delta_1 = - \check \kappa\sin \theta  \;,
\delta_2 = \check \kappa\cos \theta\,,\]
and
\[
\gamma_1 =-\zeta\sin \theta \;,\; \gamma_2 =\zeta \cos \theta+ 
\kappa_g(x_0)\, \sin \theta\;.
\]
Hence our ``normal'' form becomes $$ \widetilde b_j(r,s,t)=\widetilde b_j^0 (r,s,t)+
\Og (r^2+s^2 +t^2)$$
with 
\begin{equation}
\begin{array}{ll}
\widetilde b_1^0(r,s,t) &=  \cos \theta - (\zeta \, r +\check \kappa  s)\sin \theta   +
\sigma_1 t\;,\\
\widetilde b_2 (r,s,t)&  = \sin \theta +( \zeta \, r  +  \check \kappa
\,  s  ) \cos \theta  + 
\kappa_g(x_0) r\sin \theta + \sigma_2 t\;,\\
\widetilde b_3(r,s,t) & \equiv \gamma_3 \,r +  \sigma_3 t\;,
\end{array}
\end{equation}
with
\begin{equation}\label{mod1e}
\gamma_3=  {\kappa_{n,{\bf B}} (x_0)} = \pa_r \langle \Bb\,|\, N \rangle  \;.
\end{equation}
Now consider $\widetilde\Bb=\curl_{(r,s,t)}\tilde\Ab$. We have $\widetilde\Bb=|g|^{1/2}(\tilde b_1,\tilde b_2,\tilde b_3)$ (see \cite[Eq.~(5.13)]{HelMo4}), where $g$ is  introduced in \eqref{eq:g-ij}.  So we obtain by  \eqref{eq:det-g},
$$ \widetilde\Bb_{ij}(r,s,t)=\widetilde \Bb_{ij}^0(r,s,t) +
\Og (r^2+s^2 +t^2)$$
with 
\begin{equation}
\begin{array}{ll}
\widetilde\Bb_{23}^0(r,s,t) &=  (1-\kappa_g(x_0) r)\cos \theta   - (\zeta \, r +
\check \kappa   s)\sin \theta  +
\sigma_1 t\;,\\ 
\widetilde\Bb_{31}^0 (r,s,t)& =     \sin \theta + ( \zeta \, r+    \check \kappa
\,  s )\cos \theta + \sigma_2 t\;,\\
\widetilde\Bb_{12}^0(r,s,t) & \equiv \gamma_3 \,r +  \sigma_3 t\;.
\end{array}
\end{equation}
Notice that the condition ${\rm div}_{(r,s,t)}\widetilde\Bb=0$ reads (at $r=t=0$ and $s=0$) as follows
\[\sigma_3=\big(\kappa_g(x_0) -\check\kappa\big)\cos\theta +\zeta\sin\theta\,.\]
We have now to choose a suitable corresponding magnetic potential  to  $\widetilde\Bb^0$.
We find 
\begin{equation}\label{mod1b} 
\tilde{\Ab}^{00}(r,s,t)=\left(\begin{array}{c}  
\tilde A_1^{00}\\
 \tilde A_2^{00}\\
\tilde A_3^{00}\end{array}\right)
=
\left(\begin{array}{c}
t a_1(r,s) +\frac{\sigma_2}2t^2 \medskip\\
 ta_2(r,s)+ \frac{1}{2} \gamma_3  r^2 -\frac{\sigma_1}2t^2\medskip\\
 0\end{array}\right)=\Ab^{00}(r,s,t) +\mathcal O(t^2)\,,
\end{equation}
with 
\begin{equation}\label{mod1c}
a_1(r,s) = \sin \theta +  (\zeta 
r + \check \kappa s)\cos \theta\;,
\end{equation}
\begin{equation}\label{mod1d}
a_2(r,s) = -(1- \kappa_g(x_0)   r )\cos \theta +   (\zeta  r +\check \kappa s) \sin
\theta\;.
\end{equation}
 Moreover ${\rm  curl\,}\tilde\Ab^{(2)}=\widetilde\Bb^0$ in the  simply  connected domain $\tilde Q_h(x_0,R_0,\delta,\epsilon_2)$, so we can find a function $\check p$ such  that $\tilde\Ab^{(2)}=\tilde\Ab^{00}-\nabla \check p$.

 Finally, $\gamma_j(s):=\frac{\partial\tilde  b_j}{\partial r}(0,s,0)$ and $\delta_j(s):=\frac{\partial\tilde b_j}{\partial s}(0,s,0)$ are   bounded  functions. Setting $M_j=\sup \big(|\gamma_j(s)|+|\delta_j(s)|\big)$ and $M=\max(M_1,M_2)$, we get  from 
\eqref{eq:delta1,2}   that
$$|\check\kappa|  \leq 2M \mbox{ and } |\zeta|\leq 2M+\|\kappa_g\|_\infty\,.
$$
\end{proof}

\subsection{A second comparison estimate}

We use the magnetic potential in Lemma~\ref{lem:normal-form} to approximate the quadratic form, as we did in Lemma~\ref{lem:10.1}. In particular, we approximate the metric by a flat one.  
Let us introduce the quadratic form corresponding to the magnetic potential in Lemma~\ref{lem:10.1} (see \cite[Lem.~10.2]{HelMo4}):
\begin{multline}\label{eq:norm-form}
q_{\Ab^{00}}^h(v)=
\int_{\tilde Q_h(x_0,R_0)}\Big( |hD_tv|^2+(1+2r\kappa_g(x_0))|(hD_s-A^{00}_2)v|^2\\
+|(hD_r-A_1^{00})v|^2\Big)drdsdt,
\end{multline}
where $v\in H^1(\tilde Q_h(x_0,R_0))$ and $\tilde Q_h(x_0,R_0)=\tilde Q_h(x_0,R_0,\delta,\epsilon_2)$ is the set introduced in \eqref{eq:tilde-Q}.

 We can  obtain a further approximation of the quadratic form for functions obeying the conditions in \eqref{eq:cond-u}.

\begin{lemma}[Helffer-Morame]\label{lem:10.2}\
There exist positive constants $C,h_0,\varsigma_0$ such that, for all $h\in(0,h_0]$ and $u\in H^1(\Omega)$  s.t. ${\rm supp\,}\,u\subset Q_h(x_0,R_0,\delta,\epsilon_2)$, we have
\[
 q_{\Ab^{00}}^h(\tilde u)-CM_h(u) h^{ \frac43+\varsigma_0} \leq q_\Ab^h(u)\leq q_{\Ab^{00}}^h(\tilde u)+CM_h(u) h^{\frac43+\varsigma_0}\,,
 \]
 where $M_h(u)$ is introduced in \eqref{eq:Mh(u)} and 
 \[\tilde u=(1-r\kappa_g(x_0))^{1/2}u \,e^{-i \check p/h}.\]
\end{lemma}
\begin{proof}
We have the following two estimates from  \cite[Lem.~10.2]{HelMo4}   (whose proof  does not require that the magnetic field $\curl\Ab$ is constant)
\begin{multline*}
 q_\Ab^h(u)\geq q_{\Ab^{00}}^h(\tilde u)-C\|t^{1/2}(h D_x-\Ab)u\|^2 \\
 -C\big(q_{\Ab^{00}}^h(\tilde u)\big)^{1/2}\| (h^{3\delta}+h+h^{2\delta}t+  t^2)u\|\\
 -C\| (h^{3\delta}+h+h^{2\delta}t+  t^2) u\|^2\,,
 \end{multline*}
 and
 \begin{multline*}
 q_\Ab^h(u)\leq q_{\Ab^{00}}^h(\tilde u)+C\|t^{1/2}(h D_x-\Ab)u\|^2 \\
 +C\big(q_{\Ab^{00}}^h(\tilde u)\big)^{1/2}\| (h^{3\delta}+h+h^{2\delta}t+ t^2)u\|\\
 +C\| (h^{3\delta}+h+h^{2\delta}t+   t^2)u\|^2\,.
 \end{multline*}
 We can then estimate the remainder terms, using \eqref{eq:cond-u}, as we did in the proof of Lemma~\ref{lem:10.1}. 
  The only  term that was not present  satisfies
 \[ \|t^2u\|^2\leq M_h(u)h^2\,,\]
 where we used \eqref{eq:cond-u} (c) with $n=4$.
\end{proof}
\subsection{An estimate away from the  curve $\Gamma$}

Let us now look at the quadratic form, $q_\Ab^h(u)$, when $u$ is supported away from $\Gamma$. We start with  a rough lower bound.
\begin{lemma}\label{lem:10.1*f}
Given  $c>0$, $\epsilon_2\in(0,1)$ and $\rho\in(0,\frac14)$,  there exist positive constants $h_0,\tilde c $ such that, if $u\in  H^1(\Omega)$ satisfies
\[{\rm supp\,}\,u\subset\{x\in\Omega~:~{\rm dist}(x,\partial\Omega)<\epsilon_2\,,~d_\Gamma(x)\geq c\,h^\rho\},\]
where $d_\Gamma(x)={\rm dist}_{\partial\Omega}(p(x),\Gamma)$ is introduced in \eqref{eq:d-Gam}, then
\[q_\Ab^h(u)\geq (\Theta_0+ \tilde c\, h^{\rho})h\int_\Omega|u|^2dx\,. \]
\end{lemma}
\begin{proof}
If we verify that, for a given constant $c>0$,  
\begin{equation}\label{eq:d-Gam>c}
d_\Gamma(x)\geq ch^\rho\implies \exists\,c'>0,~|\nu(x)|\geq c'h^\rho\,, \end{equation}
then the proof follows from Proposition~\ref{prop:lb-qf}, by using  that  $h^{5/4}=o(h^{1+\rho})$ and the lower bound from Proposition~\ref{prop:gs-hp},
\[\sigma(\nu)\geq \Theta_0+\frac{\sqrt{\delta_0}}2|\nu| \,, \]
in a neighborhood  of $0$. 

Let us denote by $m_*=\min_{x\in\Gamma}\kappa_{n,\Bb}(x)$, then $m_*>0$  by Assumption~\ref{ass:C1}, and \eqref{eq:d-Gam>c}  holds with $c'=m_*c/2$.  In fact, if $|\nu(x)|\leq c'h^\rho$, we get by  \eqref{eq:def-nu} 
\[|\Bb\cdot\Nb (p(x))|\leq c'h^\rho\,, \]
and it follows from \eqref{eq:B-r,t=0} that (recall that $d_\Gamma(x)=|r|$, see Sec.~\ref{sec:coordinates})
\[ m_* d_\Gamma(x)\leq c'h^\rho=m_*\frac{c}2h^\rho \,. \]
\end{proof}
The next proposition is an improvement of Proposition~\ref{lem:10.1*f} since it allows for the  support  of $u$  to be closer to  the curve $\Gamma$.
\begin{proposition}\label{lem:10.1*}
Given  $c>0$, $\epsilon_2\in(0,1)$ and $\delta\in[\frac14,\frac13)$,  there exist positive constants $h_0,c_*,C,\varsigma_0$ such that, if $u\in  H^1(\Omega)$ satisfies 
\begin{equation}\label{eq:cond-sup-u}
{\rm supp\,}\,u\subset\{x\in\Omega~:~{\rm dist}(x,\partial\Omega)<\epsilon_2,~d_\Gamma(x)\geq c\,h^\delta\},\end{equation}
where $d_\Gamma(x)={\rm dist}_{\partial\Omega}(p(x),\Gamma)$ is introduced in \eqref{eq:d-Gam}, then
\[q_\Ab^h(u)\geq (\Theta_0+c_*h^{\delta})h\int_\Omega|u|^2dx-CM_h(u)h^{\frac43+\varsigma_0}\,, \]
where $M_h(u)$  is introduced  in \eqref{eq:Mh(u)}.
\end{proposition}
\begin{proof}~\\
{\bf Step~1.} Let us  fix constants $c,R_0>0$, $\epsilon_2\in(0,1)$, $\delta\in[\frac14,\frac13)$  and  $\rho\in(0,\frac14)$. We assume that ${\rm supp\,}\,u\subset Q_h(x_0^*,R_0,\delta,\epsilon_2)$ where $x_0^*\in\partial\Omega $ with  boundary coordinates $(r_0,s_0,t_0=0)$ satisfies (for $h$ small enough)
$ c\,h^\delta\leq  |r_0|=d_\Gamma(x_0^*)\leq 2c\,h^\rho$  
and $Q_h(x_0^*,R_0,\delta,\epsilon_2)$ is  introduced in  \eqref{eq:Q(x0)}. 

 We denote by $\tilde Q_h(x_0^*)=\tilde Q_h(x_0^*,R_0,\delta,\epsilon_2)$ the neighborhood associated with $Q_h(x_0^*,R_0,\delta,\epsilon_2)$ by  \eqref{eq:tilde-Q}. By a translation, we may  assume that $s_0=0$.

Consider the magnetic  potential  $\tilde\Ab^{(2)}$  introduced in \eqref{eq:tilde-A-2}.  We modify the coordinates $(r,s,t)$  so that, locally  near $(r_0,0,0)$, the metric $G$ in \eqref{eq:alpha(r,s)} is diagonal\footnote{We consider the  curve $\Gamma_h$ defined by $s\mapsto \Phi_{x_0}^{-1}(r_0,s,0)$, where $x_0=\gamma(x_0^*)$ and $\Phi_{x_0}$ is  the coordinate  transformation  introduced in \eqref{eq:Phi-x0}.   We   parameterization $\Gamma_h$ by  arc-length $s\mapsto \gamma_h(s)$ and define the adapted  coordinates by considering the normal  geodesic to $\Gamma_h$ passing through $x_0^*$.} with
\begin{equation}\label{eq:alpha(r0,s)=1}
 \alpha(r_0,s)=1\quad{\rm and}\quad \frac{\partial\alpha}{\partial r}(r_0,s)=-2\kappa_g(\gamma(s))+ \mathcal O(h^\rho)\,.
 \end{equation}
By   Taylor's formula
\[\alpha(r,s)=1 -2\kappa_g(\gamma(s))(r-r_0)  +\mathcal O(h^\rho(r-r_0))+\mathcal O ((r-r_0)^2)\,.\]
In $\tilde Q_h(x_0^*)$,  we write  
\[|\kappa_g(\gamma(s))-\kappa_g(x_0^*)|\leq Ch^\delta,   \]
\[ \begin{aligned}
\alpha(r,s)=1 -2\kappa_g(x_0^*)(r-r_0)-Ch^{\delta+\rho}
 \end{aligned}\,, \]
and
\[hD_y-\tilde\Ab=(hD_y-\tilde\Ab^{(2)})-(\tilde\Ab-\tilde\Ab^{(2)})\,. \] 
So we get, as in  Lemma~\ref{lem:10.1}, the  existence  of $C',\varsigma_0>0$  such that
\[q_\Ab^h(u)\geq (1-Ch^{\delta+\rho})q_{\tilde\Ab^{(2)}}^{h,x_0^*}(u)-C'M_h(u)h^{\frac43+\varsigma_0}\,, \]
where
\begin{multline*}
q_{\tilde\Ab^{(2)}}^{h,x_0^*}(u) = \int_{\tilde Q_h(x_0^*)}(1-(r-r_0)\kappa_g(x_0^*))\Big(|(hD_t-\tilde A_3^{(2)})u|^2\\
+(1+2(r-r_0)\kappa_g(x_0^*))|(hD_s-\tilde A_2^{(2)})u|^2+|(hD_r-\tilde A_1^{(2)})u|^2\Big)drdsdt\,.
\end{multline*} 
  Performing a  change of variables 
\[(r,s)\mapsto  \big((r-r_0)\cos\omega -s\sin\omega,(r-r_0)\sin\omega+s\cos\omega\big)  \] 
which amounts to a rotation in the $(r,s)$-plane  (centered at  $(r_0,0)$), we may  assume that the  second  component of $\widetilde\Bb={\rm  curl}_{(r,s,t)}\tilde\Ab=(\tilde B_{23},\tilde B_{31},\tilde B_{12} )$  vanishes  at $(r_0,0,0)$, by choosing $\omega$ so that 
\[\tilde B_{31}(x_0^*)\cos\omega+ \tilde B_{23}(x_0^*)\sin\omega=0\,.
\]
At the  same  time, this rotation  leaves $|\Bb|$ and the measure $drds$ invariant.  
Then  performing  a gauge transformation (see \cite[Sec.~16.3]{HelMo4}), we may  assume that
\[ \tilde\Ab^{(2)}(r,s,t)=\tilde\Ab^{(2,0)}(r,s,t)+\mathcal O(|r-r_0|t+|s|t+t^2) \]
where  
\[\tilde\Ab^{(2,0)}(r,s,t):=\left(\begin{array}{l}
\tilde c_1^0\, s^2\\
\tilde B_{23}^{(0)}t+\tilde B_{12}^{(0)}(r-r_0)+\tilde c_2^0\, (r-r_0)^2\\
0
\end{array} \right) \,.\]
Here $$ \widetilde\Bb^{(0)}:=\widetilde\Bb(r_0,0,0)=(\tilde B_{23}^{(0)},\tilde B_{31}^{(0)}=0,\tilde B_{12}^{(0)} )$$ and $\tilde  c_1^0,\tilde  c_2^0$ are constants. \\
 Similarly to  the proof of Lemma~\ref{lem:10.1},  by  writing
\[hD_y-\tilde\Ab^{(2)}=hD_y-\tilde\Ab^{(2,0)}-(\tilde\Ab^{(2)}-\tilde\Ab^{(2,0)} )\]
and
\[\|(hD_y-\tilde\Ab^{(2,0)})u\|\leq \|(hD_y-\tilde\Ab)u\|+\|(\tilde\Ab-\tilde\Ab^{(2)})u\|+\|(\tilde\Ab^{(2)}-\tilde\Ab^{(2,0)})u\|\,,\] 
we get
\[ q_{\tilde\Ab^{(2)}}^{h,x_0^*}(u) \geq (1-Ch^{2\delta})q_{\tilde\Ab^{(2,0)}}^{h,x_0^*}(u)-C''M_h(u)h^{\frac43+\varsigma_0}\,.\]
Thus we are left with finding a lower bound of $q_{\tilde\Ab^{(2,0)}}^{h,x_0^*}(u)$. \\
 Note  that,
since  $|\Bb|=1$  and by  \eqref{eq:alpha(r0,s)=1}, the metric satisfies $|g|=1$ on  $x_0^*$,    we have  by \eqref{eq:normB},  $|\tilde B_{23}^{(0)}|^2+|\tilde B_{12}^{(0)}|^2=1$.\\
 Moreover, since $\Bb\cdot\Nb$  vanishes linearly on $\Gamma=\{r=0\}$,  there exist $C_1>0$ and $C_2>0$ such that
\[ \frac 1{C_1} |r_0|\leq |\tilde B_{12}^{(0)}|^2\leq C_2\, |r_0|\,,\quad |\,|\tilde B_{23}^{(0)}|-1\,|\leq C_2r_0^2,\quad
|\tilde c_1^0| +|\tilde c_2^0|\leq C_2  \,. \]
The previous  estimates yield a lower bound of $q_{\tilde\Ab^{(2,0)}}^h(u)$ by comparing  with a  model operator (after rescaling the variables $\tilde r=h^{1/3}(r-r_0)$, $\tilde s=h^{1/3}s$ and  $\tilde t=h^{1/2}t$).   In fact, by \cite[Lemma~16.1]{HelMo4}, there exists $c_1>0$ such that,
\[ q_{\tilde\Ab^{(2,0)}}^{h,x_0^*}(u)\geq (\Theta_0+c_1|r_0|) h \int_{\Omega}|u|^2dx\,.\] 
Note that, we can use Lemma~16.1 of \cite{HelMo4}  under our assumptions on the support  of  $u$.\medskip\\
{\bf Step~2.} We can reduce to the setting  of Step~1 and Lemma~\ref{lem:10.1*f} by means of a partition of unity.  In fact, consider an $h$-dependent  partition of unity $\chi_1^2+\chi_2^2=1$ on  $\{{\rm dist}(x,\partial\Omega)<\epsilon_2\}$ such  that  
\[{\rm supp\,}\chi_1\subset\{d_\Gamma(x)\geq \frac{c}2h^\rho \},\quad {\rm supp\,}\chi_2\subset \{d_\Gamma(x)\leq c h^\rho \},\quad\sum_{i=1}^2|\nabla\chi_i|^2=\mathcal  O(h^{-2\rho})\,.  \]
If $u\in  H^1(\Omega)$ satisfies \eqref{eq:cond-sup-u}, then
\[ q_\Ab^h(u)=\sum_{i=1}^2\Big(q_\Ab^h(\chi_iu)-h^2\|\,|\nabla\chi_i|u\, \|^2\Big)\,,\]
where 
\begin{align*}
&q_\Ab^h(\chi_1u)\geq (\Theta_0+ \tilde c\, h^{\rho})h\int_\Omega|\chi_1u|^2dx\quad{\rm by~Proposition~\ref{lem:10.1*f}}\,,\\
&q_\Ab^h(\chi_2u)\geq (1-Ch^{\delta+\rho})(\Theta_0+  c_1 h^{\delta})h\int_\Omega|\chi_2u|^2dx-M_h(u)h^{\frac43+\varsigma_0}\quad{\rm by~Step~1}\,,\\
&\sum_{i=1}^2h^2\|\,|\nabla\chi_i|u\, \|^2=\mathcal O(h^{2-2\rho})=o(h^{1+\delta}) \,,
\end{align*}
where in  the  last step we used that  $0<\rho<\frac14$  and $\frac14<\delta<\frac13$.
\end{proof}

\section{Lower bound}\label{sec:lb}

\subsection{Another model}

The  model in \eqref{n1a} corresponds to  the quadratic form in \eqref{eq:norm-form} when $\kappa_g(x_0)=0$. However, when $\kappa_g(x_0)\not=0$, the  situation is similar to \cite[Sec.~15]{HelMo4}. The model  compatible with \eqref{eq:norm-form} can still be reduced to  the one in \eqref{n1a} with  appropriate choices  of the  parameters $\eta,\zeta,\gamma$ (see \eqref{eq:def.eta.zeta}).

\subsubsection{A new model  quadratic form}

Let us fix a boundary point $x_0\in\Gamma$ and denote the model quadratic  form  near $x_0$ by 
\begin{equation}\label{eq:qm}
u \mapsto q_m(u):=q_{\Ab^{00}}(u)
\end{equation}
where $q_{\Ab^{00}}$ is given in \eqref{eq:norm-form}, $u\in H^1(\tilde Q_h(x_0,R_0))$ and $\tilde Q_h(x_0,R_0)= \tilde Q_h(x_0,R_0,\delta,\epsilon_2)$ is the set introduced in \eqref{eq:tilde-Q}. 
Furthermore, we assume that
 that the metric is flat at $x_0$ and the coordinates of $x_0$ in the $(r,s,t)$ frame are $(0,s_0=0,0)$, after performing a translation  with respect to  the $s$ variable.

Following the proof of  \cite[Lem.~15.1]{HelMo4}, we are led  to the analysis of the 
 model quadratic form  (see Lemma~\ref{lem:qm=qm0})
\begin{equation}\label{mlem1eq1} 
q^{h}_{m,0}(u)=
\int_{\tilde Q_h(x_0,R_0)}\Big(h^2|D_tu|^2  +
|tu-L^h_1u|^2 + 
 |L^h_2u|^2 \Big)\; drdsdt \; , 
\end{equation}
where
\begin{equation}\label{mod1bb} 
\begin{aligned} 
L^h_1&=a_1\, hD_r+a^0_2\, hD_s- \frac{1}{2}\,\cos \theta\, \kappa_{n,\bf B}(x_0)\, r^2
\; , \\ 
L^h_2&= a^1_2\, hD_r+ a^1_1\, hD_s + \frac{1}{2}\,\sin \theta \,  \kappa_{n,\bf B}(x_0)\, 
r^2
\;  ,
\end{aligned} 
\end{equation} 
and,  with $\theta=\theta(s_0)$ the angle  defined by \eqref{eq:theta-r,t=0}, we introduce the following  functions 
\begin{equation}\label{eq:def.a-alpha(r)}
\begin{array}{ll}
a_1(r,s) &= \sin \theta + \cos \theta (\zeta 
r + \check \kappa s)\;,\\
a_2(r,s) & = -\cos \theta + \kappa_g(x_0)  \cos \theta r +  \sin
\theta (\zeta  r +\check \kappa s) \;,\\
a_{2}^0(r,s)& =-\cos \theta - \kappa_g(x_0)  \cos \theta r+ \sin \theta (\zeta r
+\check \kappa s)\;,\\
a_2^1(r,s) &= \cos \theta - \sin \theta (\zeta r + \check \kappa s) \;,\\
a_1^1(r,s)& = \sin \theta + \sin \theta \kappa_g(x_0)  r + \cos \theta (\zeta r
+\check \kappa s) \;,\\
\alpha (r)&= 1+ 2 \kappa_g(x_0)  r\;.
\end{array}
\end{equation}
We will consider the form $q_{m,0}$ on the following class of  functions
\begin{equation}\label{eq:domD0}
\mathcal D_0=\{u\in H^1(\Omega^h)~:~u|_{(\partial\mathcal Q^h)\times\,]0,h^\delta[}=0,~u|_{\mathcal Q^h\times\{h^\delta\}}=0 \}
\end{equation}
where
\[ \Omega_h=\mathcal Q^h\times\,]0,h^\delta[\,,\quad \mathcal Q^h=]-R_0h^\delta,R_0h^\delta[^2\,.\]
The precise relation between the model quadratic forms in \eqref{eq:qm} and \eqref{mlem1eq1} is given in the  following lemma.
\begin{lemma}\label{lem:qm=qm0}
 For any $\delta\in(\frac{5}{18},\frac13)$ and $\tau_1 > 0$, there exists $C >0$ such that, for any $u\in \mathcal D_0$ and $h\in(0,1)$,
\[(1 +Ch^{2\delta})q^h_m(u)
\geq (1-Ch^{\tau_1})q^h_{m,0}(u)-C\big(
\|(h^{2\delta}+h^{\tau_1})tu\|^2+
h^{6\delta-\tau_1}\|u\|^2\big)\,.
\]
\end{lemma}
\begin{proof}
The proof follows that of Lemma~15.1 in \cite{HelMo4} with some adjustments  in the formulas (15.9),
(15.16)
 and (15.17) in \cite{HelMo4}.\\
We have indeed 
\begin{equation}\label{1.2a} 
|1-(a_1)^2 -\alpha (a_2)^2|\leq Ch^{2\delta}\;, 
\end{equation} 
 where we  used that  $\alpha(r)^{1/2}=1+\kappa_g(x_0)r +\mathcal O(h^{2\delta})$ on the support of $u$, which  follows by \eqref{eq:def.a-alpha(r)}.

We also  observe that~: 
\begin{equation}\label{1.7} 
 |\alpha a_2-a^0_2|+|\alpha ^{1/2}a_2+ a^1_2|+
|\alpha ^{1/2}a_1-a^1_1|\leq C(r^2+s^2) 
\end{equation} 
and 
\begin{equation}\label{1.8} 
|\alpha ^{1/2}a_1-\sin \theta |+|\alpha a_2+\cos \theta |\leq 
C(r^2+s^2)^{1/2} \;. 
\end{equation}
\end{proof}
 Later  on, we will choose $\delta$  and $\tau_1$  in a convenient way (see Remark~\ref{rem:tau1}).
\subsubsection{Linearizing change of variable}

In order to reduce to  the case $\kappa_g=0$  and eliminate the slightly variable
 coefficients of $D_r$ and $D_s$ in \eqref{eq:norm-form}, we argue as \cite[Sec.~15.2]{HelMo4a} by performing a change of variables. The argument does not work in our case in  the same way as \cite[Sec.~15.2]{HelMo4a},
but it leads to the fact that for our lower bound the only relevant
 parameters  are $\eta:=\check\kappa - \kappa_g$  and $\zeta$ (see \eqref{eq:norm-form}). 
 
  The  below computations  are essentially the same as in \cite[Sec.~15.2]{HelMo4} but we have to do them  carefully in order to capture the correct  $\eta$ and $\zeta$ appearing in \eqref{n1a}.
 
 Let us follow, what
this change of variable was doing.
 We introduce 
\begin{equation}\label{hatkappa}
\kappa :=\kappa_g(x_0) .
\end{equation}
Let us make the change of variables $(r,s)=\Phi_{\kappa} (p,q)$ with 
\begin{equation}\label{varchange} 
\begin{array}{l} 
r= \sin \theta \,  p +\cos \theta \,  q 
-\frac{\kappa}{2}[-\cos \theta \,  p +\sin \theta \,  q ]^2\; ,\\  
s=-\cos \theta  \, p + \sin \theta \,  q 
-\frac{\kappa}{2}[\sin (2\theta )\, (p^2 -q^2) + 2\cos (2\theta )\, pq] \; 
, \end{array} 
\end{equation} 
 where $\theta=\theta(s_0)$ is the angle  defined by \eqref{eq:theta-r,t=0}.

The map $\Phi_{\kappa}$ is a perturbation of a rotation 
and, by the local inversion theorem, it is easily seen as a local diffeomorphism
 sending a fixed neighborhood of $ (0,0)$ onto another
neighborhood of $(0,0)$.

Then, for $h$ small enough, $\mathcal Q^h:=]-R_0h^\delta,R_0h^\delta[^2$ is transformed   by $\Phi_{ \kappa}^{-1}$ to the set $ 
\mathcal Q^h_0$
 satisfying~:
\begin{equation}\label{varchange2} 
\mathcal  Q^h_0=\Phi_{ \kappa} ^{-1}(\mathcal Q^h) \; \subset \, \; 
]-R_0'h^{\delta }\;,\; R_0'h^{\delta }[ \; \times \; 
 ]-R_0'h^{\delta }\;,\; R_0'h^{\delta }[ \; . 
\end{equation} 
 Let us write
\begin{equation}\label{varchange3} 
D_p=c_{11}D_r+ c_{12}D_s,\ \ \ D_q =  c_{21}D_r+ c_{22}D_s \;,
\end{equation} 
 We can express the   functions $c_{ij}$ in terms of  the $(p,q)$ variables, by using  (\ref{varchange}). In fact, we introduce $c_{ij}(r,s)=\check  c_{ij}(p,q)$, and observe that
\begin{align*}
\check c_{11}(p,q) & = \frac{\pa r}{\pa p} = \sin \theta +  \kappa \cos \theta \, 
(-\cos \theta\, p + 
\sin \theta \,q)\;;\\
\check c_{12}(p,q) & = \frac {\pa s }{\pa p}=- \cos \theta -  \kappa \, (\sin( 2 
\theta)\, p + \cos 
(2\theta)\, q)\;;\\
\check c_{21}(p,q) &=\frac{\pa r}{\pa q}  = \cos \theta -  \kappa \sin \theta \, (-\cos 
\theta \,p + 
\sin \theta\,  q)\;;\\
\check c_{22}(p,q) &= \frac{\pa s}{\pa q} = \sin \theta- \kappa \, (- \sin (2\theta)\, 
q + \cos 
(2\theta)\,  p)\;.
\end{align*}
  Then we return back to  the $(r,s)$  variables, by using  (\ref{varchange}).  Noticing that, as $(p,q)\to(0,0)$,
 \begin{equation}\label{eq:p,q=r,s}
  r=\sin\theta\,p+\cos\theta\,q  +\Og (p^2+q^2)\,,\quad s=-\cos \theta\, p + 
\sin \theta \,q + \mathcal O(p^2+q^2)\,, \end{equation}
we get  
\begin{equation}\label{varchange4} 
\begin{aligned}
c_{11}(r,s) & = \sin \theta  +\kappa \cos \theta \, s + \Og(r^2 + s^2)\;;\\
c_{12} (r,s) & =  -\cos \theta  -\kappa  (\cos \theta\, r - \sin \theta \, s)
+ \Og (r^2+s^2)\; ;\\
c_{21} (r,s)  & =\cos \theta - \kappa \sin \theta\, s  + \Og(r^2 + s^2)\;;\\
c_{22} (r,s)  & = \sin \theta  + \kappa (\sin \theta \, r 
+  \cos \theta \, s)
 + \Og(r^2 + s^2)\;.
\end{aligned}
\end{equation} 
Let us now control the measure in the change of variable. 
By an easy computation, we get~:  
\[dr\, ds=\check \alpha_1dpdq\,,\quad \check\alpha_1(p,q)=1+\kappa (\sin\theta\, p+\cos\theta\, q)+\mathcal O(p^2+q^2)\,. \]
By  using \eqref{eq:p,q=r,s},  $\alpha_1(r,s)=\check\alpha_1 (p,q)$ satisfies
\begin{equation}\label{varchange6} 
 |\alpha_1-1- \kappa r|\leq 
C(r^2+s^2) 
\; ,
\end{equation}  
where $r=r(p,q)$ is defined in \eqref{varchange}.

\vskip 0.5cm 
\noindent 
Similarly
 to  Lemma~\ref{lem:qm=qm0} we get also  that one can go from the control
 of $q_{m,0}^h (u)$ to the control of the new  quadratic form\footnote{We express $L_1^h$ and $L_2^h$ (see \eqref{mod1bb}) in terms of the $(p,q)$ variables introduced in \eqref{varchange}  and neglect the terms of order $\mathcal O(r^2+s^2)=\mathcal O (p^2+q^2)$.}
\begin{equation}\label{mlem2eq2} 
q^{h}_{m,1}(u)=
\int_{\Omega ^h_0}\Big(h^2|D_tu|^2  +
|tu-M^h_1u|^2 + 
 |M^h_2u|^2 \Big)
\check\alpha_1\, dpdqdt \; , 
\end{equation} 
with
\[ \Omega ^h_0:=\mathcal Q^h_0\times ]0,h^{\delta }[\; ,\]
and 
\begin{equation}\label{mod2bb} 
\begin{aligned}
M^h_1&=hD_p + h \left((\check \kappa - \kappa)s + \zeta r\right) D_q -\frac{1}{2}\, \cos \theta \,\kappa_{n,\bf B}(x_0)
(\sin \theta\, p+\cos \theta \,  q)^2
\; , \\ 
M^h_2&=hD_q -  h\left((\check \kappa - \kappa)s + \zeta r\right) D_p +  \frac{1}{2}\, \sin \theta \kappa_{n,\bf B}(x_0)
(\sin \theta \, p\;+ \; \cos \theta  q)^2 
 \; ,
\end{aligned} 
\end{equation}  
 where $(r,s)=(\sin\theta p+\cos\theta q,-\cos\theta  p+\sin\theta q)$. 

More precisely, we have the following comparison lemma  (see Lemma~\ref{lem:qm=qm0} and \cite[Lem.~15.4]{HelMo4}).
\begin{lemma}\label{lem:qm0=qm1}
For any $\tau_1 > 0$, there exists $C >0$ such that, for any $u\in \mathcal D_0$,
\[(1 +Ch^{2\delta})q^h_{m,0}(u)
\geq (1-Ch^{\tau_1})q^h_{m,1}(\tilde u)-C\big(
\|(h^{2\delta}+h^{\tau_1})tu\|^2+
h^{6\delta-\tau_1}\|u\|^2\big)\,,
\]
where $\tilde u=u\circ\Phi_\kappa^{-1}$ is associated with $u$ by the transformation $\Phi_\kappa$.
\end{lemma}
By a unitary transformation, and after control of a commutator,  we
 can reduce  to a  flat measure  ($dpdq$ instead of $\check\alpha_1dpdq$) and  obtain  the  new quadratic form defined as follows
\begin{equation} \label{eq:qm2}
q^{h}_{m,2}(v)= 
\int_{\Omega ^h_0}[h^2|D_tv|^2  +
|tv-M^h_1v|^2 + 
 |M^h_2v|^2 ]\; dpdqdt\;,
\end{equation} 
 with $v$ associated to $u$ by $v=\check\alpha_{1}^{1/2}\tilde u$. In  fact, we have \cite[Eq.~(15.29)]{HelMo4}
 \begin{equation}\label{eq:qm1=qm2}
(1+Ch^{1/2}) q_{m,1}(\tilde u)+Ch^{3/2}\|u\|^2\geq  q^h_{m,2}(v)\,.
 \end{equation}

Let us consider the new model   associated with   the quadratic form in \eqref{eq:qm2}. We first observe that the result
 depends only on
  $\check\kappa -\kappa $ and on $\zeta$. The proof
 is moreover uniform with respect to these parameters.
As a consequence, if $\Phi= \Phi_{\kappa}$ was the transformation  introduced in \eqref{hatkappa}, the inverse (for $\kappa=0$)
 $\Phi_{0}^{-1}$, more explicitly the transformation
 $(p,q) \mapsto (\tilde r = \sin \theta p + \cos \theta q  \;,\; \tilde s = 
-\cos \theta p + 
\sin \theta q)$ will bring 
 us (in the new variables $(\tilde r, \tilde s , t)$) to the initial model with 
$ \kappa_g $ 
replaced by $0$, and $\check\kappa$ replaced by $\check \kappa - \kappa_g(x_0) $.
This can also be done by explicit computations.\\

Doing the transformations backwards, we are led
 to a magnetic Laplacian computed with a trivial metric
$\kappa_g=0$   but with a new magnetic potential 
\begin{equation}
a_1(r,s)^{\rm new} = \sin \theta +\cos \theta (\zeta r +
 (\check \kappa - \kappa_g(x_0) ) s)
\;,
\end{equation}
and
\begin{equation}
a_2(r,s)^{\rm new} = -\cos \theta + \sin \theta (\zeta r +
 (\check \kappa - \kappa_g(x_0) ) s)\;.
\end{equation}

So the new model is not as simple as in the uniform magnetic field case (where $\check\kappa=\kappa_g$) but it is the model
in \eqref{n1a}, which 
we have studied in the previous section with
\begin{equation}\label{eq:def.eta.zeta}
\eta = \check \kappa -\kappa_g(x_0) ,\quad  \gamma=\kappa_{n,\Bb}(x_0)\;.
\end{equation}
In fact, since $v$ is supported in $\Omega_0^h$, we have, 
\begin{equation}\label{eq:qm2-P}
q_{m,2}^h(v)=\langle P_{0;\gamma,\theta}^{h,\eta,\zeta}v,v\rangle,
\end{equation}
where  $P_{0;\gamma,\theta}^{h,\eta,\zeta}$ is the operator  in \eqref{n1a}. 
\begin{remark}\label{rem:tau1}
We will choose  $\tau_1$ in such  a manner that 
$\frac13<\tau_1<6\delta-\frac43$. This choice is possible when $\delta$ satisfies $\frac5{18}<\delta<\frac13$.
\end{remark}

\subsubsection{Conclusion}

We can now write  a lower bound for the quadratic form $ q^h_{\Ab^{00}}(u)$ in \eqref{eq:norm-form},   assuming that  $u\in H^1(\tilde Q_h(x_0,R_0))$ and $\tilde Q_h(x_0,R_0)$ is the set introduced in \eqref{eq:tilde-Q}.
 Let $\frac5{18}<\delta<\frac13$ and $\frac13<\tau_1<6\delta-\frac43$. Collecting Lemmas~\ref{lem:qm=qm0}, \ref{lem:qm0=qm1},  \eqref{eq:qm1=qm2}, \eqref{eq:qm2-P} and Proposition~\ref{lemestmodzone2bis}, we get the existence of positive constants $C$ and $\varsigma_0$, such that 
\begin{multline}\label{eq:con-norm-form}
q^h_{\Ab^{00}}(u)\geq \big(h\Theta_0+h^{\frac 43} c^{\rm conj}\big(\theta,\kappa_{n,\Bb}(x_0)\big) -
 C h^{\frac 43+\varsigma_0} \big)\| u\| ^2\\
 -C h^{\frac 13+\varsigma_0}\|t u\| ^2-C\|(h^{2\delta}+h^{\tau_1})tu\|^2
\end{multline}
where  $c^{\rm conj}(\gamma, \theta)$ is introduced in Proposition~\ref{lemestmodzone2} with $\theta=\theta(s_0)$ the angle  in \eqref{eq:theta-r,t=0}.

\subsection{The general  case}

We return  now to  the proof of the asymptotics of the lowest eigenvalue, $\lambda_1^N(\Ab,h)$, of the operator $P_\Ab^h$  in \eqref{eq:op}. Under Assumptions (C1)-(C2), we will prove  the following lower bound:
\begin{equation}\label{eq:gse-lb-3D} 
\lambda_1^N(\Ab,h)\geq  \Theta_0 h + \widehat{\gamma}_{0,\Bb}  h^{\frac{4}{3}} + {\mathcal O}(h^{\frac{4}{3} + \eta_*}),
\end{equation}
for some constant $\eta_*>0$, where $\widehat\gamma_{0,\Bb}$ is introduced in \eqref{eq:hat-gam}.

Let $u_h$ be a normalized ground state  of $P_\Ab^h$ , i.e.
\[ \lambda_1^N(\Ab,h)=q_\Ab^h(u_h)=\|(h\nabla-i\Ab)u_h\|^2\,.\]
Consider $\frac5{18}<\delta<\frac13$ and the following neighborhood of the curve $\Gamma$,
\begin{equation}\label{eq:Gam-ep-h}
\Gamma_\delta^h=\{x\in\Omega~:~{\rm dist}(x,\partial\Omega)<h^\delta,~{\rm dist}_{\partial\Omega}(x,\Gamma)<h^{\delta/2}\}\,.
\end{equation}
In terms  of the $(r,s,t)$  coordinates introduced in Sec.~\ref{sbsec:coord},
\[\Gamma_\delta^h=\{0<t<h^\delta,~h^{\delta/2}<r<h^{\delta/2}\}\,.\]
Let $\chi_h\in C_c^\infty(\Gamma_\delta^h;[0,1])$  be a smooth function such that
\[\chi_h=1~{\rm on ~}\Gamma_{\delta,0}^h=\{x\in\Omega~:~{\rm dist}(x,\partial\Omega)<\frac12h^\delta,~{\rm dist}_{\partial\Omega}(x,\Gamma)<\frac12h^{\delta/2}\}\]
and
\[ |\nabla\chi_h|=\mathcal O(h^{-\delta/2}).\]
We introduce the function
\begin{equation}\label{eq:gs-wh}
w_h=\chi_hu_h\,.
\end{equation}
By Proposition~\ref{prop:dec-gs}, the eigenfunction $u_h$ is exponentially small outside $\Gamma_\delta^h$, since by our choice of $\delta$ we have  $h^{\delta/2}\gg h^{1/4}$  and $h^\delta\gg h^{1/2}$. So we have
\begin{equation}\label{eq:qf-wh}
\lambda_1^N(\Ab,h)=q_\Ab^h(u_h)=q_\Ab^h(w_h)+\mathcal O(h^\infty),\quad \|u_h\|=\|w_h\|+\mathcal O(h^\infty)\,.
\end{equation}
Consider now a partition of unity of $\R^3$
\[ \sum_{j\in\Z^3}|\chi_j|^2=1,\quad \sum_{j\in\Z^3}|\nabla\chi_j|^2
<\infty,\quad {\rm supp\,}\chi_j\subset j+[-1,1]^3\,, \]
and introduce the following functions
\begin{equation}\label{eq:gs-wh*}
w_{h,j}=\chi_{j,\delta}(x)w_h(x),\quad \chi_{j,\delta}(x)=\chi_j(h^{-\delta}x)\,.
\end{equation}
We can decompose   the quadratic form $q_\Ab^h(w_h)$ as follows
\begin{equation}\label{eq:decomp-qf-wh}
q_\Ab^h(w_h)=\sum_{j\in\mathcal J_h}q_\Ab^h(w_{h,j})+\mathcal O(h^{2-2\delta}),
\end{equation}
where
\begin{equation}\label{eq:ind-Jh}
\mathcal J_h=\{j\in\Z^3~:~{\rm supp\,}\chi_{j,\delta}\cap\Omega\not=\emptyset\}\,.
\end{equation}
Let $C_1>0$ be a fixed constant that we will choose later to be sufficiently large. We will estimate the energy $q_\Ab^h(w_{h,j})$ when the support of $w_{h,j}$ is near  the curve $\Gamma$, or away  from $\Gamma$, independently. So we introduce the sets of  indices
\begin{equation}\label{eq:ind-Jh*}
\begin{aligned}
\mathcal J_h^1&=\{j\in\mathcal J_h~:~{\rm dist}\big({\rm supp\,}\chi_{\gamma,\delta},\Gamma\big)\leq C_1h^\delta\}\\
\mathcal J_h^2&=\{j\in\mathcal J_h~:~{\rm dist}\big({\rm supp\,}\chi_{\gamma,\delta},\Gamma\big)\geq C_1h^\delta\}
\end{aligned}\,.
\end{equation}
By Proposition~\ref{lem:10.1*},
\begin{equation}\label{eq:lb-J2-s1}
 \sum_{j\in\mathcal J_h^2}q_\Ab^h(w_{h,j})\geq \sum_{j\in\mathcal J_h^2}\Big((\Theta_0h+c_*h^{1+\delta})
 \|w_{h,j}\|^2-Ch^{\frac43+\varsigma_0}M_h(w_{j,h}\Big)\,,\end{equation}
 where $M_h(w_{j,h})$ is introduced in \eqref{eq:Mh(u)}.  Notice  that
\begin{multline*} M_h(w_{j,h})\leq
 \sum_{n=0}^6 h^{-n/2}\int_\Omega t(x)^n
\Big(|\chi_{j,h}u_h|^2+2h^{-1}|\chi_{j,h}(h\nabla-i\Ab)u_h|^2\\+2h|\nabla(\chi_h\chi_{j,h})|^2|u_h|^2\Big)dx\,.\end{multline*}
Since $\sum |\chi_{j,h}|^2\leq 1$ and $\sum |\nabla(\chi_{j,h}\chi_h)|^2=\mathcal O(h^{-2\delta})$, Proposition~\ref{prop:dec-gs} together with  \eqref{eq:dec-gs1} and \eqref{eq:dec-gs2} yield
 \[
 \sum_{j\in\mathcal J_h }M_h(w_{j,h})=\mathcal  O(1)\,.
 \]
 Consequently, we infer from \eqref{eq:lb-J2-s1},  
\begin{equation}\label{eq:lb-J2}
 \sum_{j\in\mathcal J_h^2}q_\Ab^h(w_{h,j})\geq (\Theta_0h+c_*h^{1+\delta})\Big(\sum_{j\in\mathcal J_h^2} 
 \|w_{h,j}\|^2\Big)-C'h^{\frac43+\varsigma_0}\,.\end{equation} 
For $j\in\mathcal J_h^1$, we estimate $q_\Ab^h(w_{h,j})$ by collecting \eqref{eq:con-norm-form} and the estimates in Lemma~\ref{lem:10.1} and \ref{lem:10.2}.  We start  by picking  $R_0>0$ and $x_0^j\in\Gamma$,  so that 
\[ {\rm supp\,}\,w_{h,j}\subset Q_h(x_0^j)\]
where $Q_h(x_0^j)$ is introduced in \eqref{eq:Q(x0)}. 
Eventually, we find
\[ \sum_{j\in\mathcal J_h^1}q_\Ab^h(w_{h,j})\geq \sum_{j\in\mathcal J_h^1}(\Theta_0h+ h^{4/3}c^{\rm conj}(\theta_j, \kappa_{n,\Bb}(x_j^0)) )\|w_{h,j}\|^2-Ch^{\frac43+\varsigma_*}\,,\]
for some constant $\varsigma_*>0$, where
\[\theta_j=\theta(s_0^j)\]
and $(0,s_0^j,0)$ denote the coordinates of $x_0^j$ in the $(r,s,t)$-frame (see Sec.~\ref{sec:coordinates} and Eq.~\ref{eq:Phi-1-x0}).  Note that  we used Proposition~\ref{prop:dec-gs}  to control the term $\sum_{j\in\mathcal J_h^1}\|t w_{h,j}\|^2$ appearing in \eqref{eq:con-norm-form}; in fact $\sum_{j\in\mathcal J_h^1}\|t w_{h,j}\|^2=\mathcal O(h)$.

 Since $c^{\rm conj}(\theta_j, \kappa_{n,\Bb}(x_j^0))$  is bounded from below by  $\widehat{\gamma}_{0,\Bb}$ (see \eqref{eq:hat-gam}),  we get
\begin{equation}\label{eq:lb-J1}
 \sum_{j\in\mathcal J_h^1}q_\Ab^h(w_{h,j})\geq (\Theta_0h+\widehat{\gamma}_{0,\Bb} h^{4/3} )\sum_{j\in\mathcal J_h^1}\|w_{h,j}\|^2-Ch^{\frac43+\varsigma_*}\,.\end{equation}
Inserting \eqref{eq:lb-J2} and \eqref{eq:lb-J1} into \eqref{eq:decomp-qf-wh},  and using \eqref{eq:qf-wh},  we deduce the lower bound in \eqref{eq:gse-lb-3D},  since $\frac5{18}<\delta<\frac13$.

\section{Upper bound}\label{sec:ub}
Fortunately,   the same quasi-mode constructed in \cite[Sec.~12]{HelMo4} (see also \cite{Pan8} for a different  formulation) yields an upper bound of the lowest eigenvalue $\lambda_1(\Ab,h)$ matching with the asymptotics in Theorem~\ref{thm:main}. More precisely,  under Assumptions (C1)-(C2), we will prove that:
\begin{equation}\label{eq:gse-ub-3D} 
\lambda_1^N(\Ab,h)\leq  \Theta_0 h + \widehat{\gamma}_{0,\Bb} h^{\frac{4}{3}} + {\mathcal O}(h^{\frac{4}{3} + \eta^*}),
\end{equation}
for some constant $\eta^*>0$, where $\widehat\gamma_{0,\Bb}$ is introduced in \eqref{eq:hat-gam}.

 However, while computing the energy of the quasi-mode, we observe additional terms (not present in \cite{HelMo4})  due to the non-homogeneity of the  magnetic field. These terms are treated in  Sec.~\ref{sec:en-qm}.

\subsection{The quasi-mode}\label{sec:qm}

The construction of  the quasi-mode in \cite{HelMo4}  is  quite lengthy and involves many auxiliary functions related to the  de\,Gennes and Montgomery  models (see \eqref{eq:Harm-osc} and \eqref{eq:Mont}).  We present here the definition of the quasi-mode along with  a useful  result from \cite[Sec.~12]{HelMo4}. 

\subsubsection{Geometry and normal form}

  Select a point $x_0\in\partial\Omega$ such that the function  in \eqref{eq:tilde-gam} satisfies
\[ \widetilde\gamma_{0,\Bb}(x_0)=\widehat\gamma_{0,\Bb}\,.\]
Let us assume that the coordinates of $x_0$ in the $(r,s,t)$-frame are $(0,s_0=0,t_0)$.  The normal form of the effective 
 magnetic potential in Lemma~\ref{lem:normal-form} now becomes
 \begin{equation}\label{eq:up-A} \Ab^{00}=\left( \begin{array}{c} A^{00}_1\\
 A^{00}_2\\
 A^{00}_3\end{array}\right)=\left( \begin{array}{c} t\sin\theta+t(\zeta r+\check\kappa s)\cos\theta\\
 -t\cos\theta+rt\kappa \cos +t(\zeta r+\check\kappa s)\sin\theta+\frac12\gamma r^2\\
 0
 \end{array}\right)\,,
 \end{equation}
 where
 \begin{equation}\label{eq:dictionary}
 \theta=\theta(s_0),\quad \kappa=\kappa_g(s_0),\quad \gamma=\kappa_{n,\Bb}(x_0)\,.
 \end{equation}
 \subsubsection{Structure of the quasi-mode}
 
Consider two positive constants $C_0$ and $\delta$ such that $\frac{5}{18}<\delta<\frac13$.  Let $\chi$ be a smooth \emph{even} function,  valued in $[0,1]$,  equal  to $1$ on $[-\frac14,\frac14]$ and supported in $[-\frac12,\frac12]$.    We set
\begin{equation}\label{eq:ub-chi-h}
\chi_h(s)=c_1h^{-\delta/2}\chi(C_0^{-1}h^{-\delta}s)\,,
\end{equation}
 where $c_1=C_0^{-1/2}\left(\int_\R\chi(\sigma)^2d\sigma\right)^{1/2}$, so that $\chi_h$ is normalized  as follows,
 \[ \int_{\R} |\chi_h(s)|^2ds=1\,.\]
Our quasi-mode, $u$, is supported in the set $Q_h(x_0,R_0,\delta,\epsilon_2)$ introduced in \eqref{eq:Q(x0)} and is of  the form
 \begin{equation}\label{eq:up-def-qm-u}
  u=e^{i\check p/h}(1-r\kappa)^{-1/2}\tilde u 
  \end{equation}
 where $(r,s,t) \mapsto \check p(r,s,t) $ is the function from Lemma~\ref{lem:normal-form} and the function $\tilde u$ is of  the form
 \begin{equation}\label{eq:def-tilde-u}
 \tilde u(r,s,t)=\exp\left(-\frac{i\rho\gamma s}{h^{1/3}}\right)\exp\left( i\frac{r\sin\theta-s\cos\theta}{h^{1/2} } \xi_0\right)\chi_h(s)v(r,t)\,,
 \end{equation}
where $\xi_0=\sqrt{\Theta_0}$ is given by \eqref{eq:Th0},  $\theta$ and $\kappa$  are  introduced  in \eqref{eq:dictionary}.\\
 The choice of $\rho$ and $v$ will be specified later\footnote{ $\rho$ is defined in \eqref{sec:en-qm}. For the definition of $v$, see \eqref{eq:def-v}, \eqref{eq:ub-def-v0} and \eqref{eq:ub-def-wh}.} so that, for some constants $C,\varsigma_*>0$,  we have \cite[Eq.~(12.8)]{HelMo4}
 \begin{equation}\label{eq:HM-(12.8)}
  q^h_{M^{00}}(v)\leq \Big(\Theta_0 h + \widehat{\gamma}_{0,\Bb}  h^{\frac{4}{3}} + Ch^{\frac{4}{3} + \varsigma_*}\Big)\|v\|_{L^2(\R\times\R_+)}^2\,.
 \end{equation}
 Here $q_{M^{00}}(v)$ arises while computing the quadratic form of  the quasi-mode in \eqref{eq:up-def-qm-u}. It is defined as follows \cite[Eq.~(12.9)]{HelMo4},
 \begin{equation}\label{eq:up-def-qf-M}
 q^h_{M^{00}}(v)=\int_{\R\times\R_+}\Big(|(hD_r-M_{1}^{00}v|^2+|M_2^{00}v|^2 +|hD_tv|^2 \Big)drdt\,,
 \end{equation}
 where
 \begin{equation}\label{eq:ub-def-M}
 \begin{aligned}
M_{1}^{00} (r,t)&= \sin\theta( t-h^{1/2}\xi_0)\\
M_{2}^{00}(r,t) &=(1+2\kappa r)^{1/2}\Big(- \cos\theta ( t-h^{1/2}\xi_0) + \kappa\cos\theta rt-b\frac{\gamma}{2}(r^2-h^{2/3} \rho)\Big)\,.
\end{aligned}
\end{equation}
Notice that, by our normalization of $\chi_h$, we have
\begin{equation}\label{eq:ub-norm-tilde-u}
 \int_{\R^2\times\R_+} |\tilde u(r,s,t)|^2drdsdt=\|v\|^2_{L^2(\R\times\R_+)}\,.
 \end{equation}

 \subsubsection{Definition of the auxiliary objects}
 
 Let us recall the definition of the function $v$ and  the parameter $\rho$ given in \cite[Sec.~12]{HelMo4}.
 The function $v$  depends on $h$ and is selected in the following form (see \cite[Eq.~(12.14)]{HelMo4})
 \begin{equation}\label{eq:def-v}
 v(r,t)= h^{-5/12}v_0(\hat r,\hat t)
 \end{equation}
 where 
 \[(\hat r,\hat t)= (h^{-1/3}r,h^{-1/2}t)\,.\]
 The  function $v_0$ is selected as in \cite[Eq.~(12.28)]{HelMo4}:
  \begin{equation}\label{eq:ub-def-v0} 
  v_0(\hat  r,\hat t)=\chi(C_0^{-1}h^{-\delta+\frac13}\hat r)\chi(C_0^{-1}h^{-\delta+\frac12}\hat t)w_h(\hat r,\hat  t),
  \end{equation}
  In the sequel, we  skip  the hats from  the notation.    The  function $w_h$ is defined as follows\footnote{ For the convenience of  the reader, we will recall the heuristics behind the construction of $w_h$ in Subsection~\ref{sec:up-def-wh}.}  \cite[Eq.~(12.22)]{HelMo4}
\begin{equation}\label{eq:ub-def-wh}
 w_h(r,t)=\varphi_0(t)\psi(r)+h^{1/6}\varphi_1(t)L_1^0(r,D_r)\psi(r)+h^{1/3}\varphi_2(t)\big(L_1^0(r,D_r)\big)^2\psi(r)\,, 
 \end{equation}
 where $\varphi_0$ is the positive normalized ground state of the harmonic oscillator in  \eqref{eq:Harm-osc},
 \[ \varphi_1(t)=2\mathcal R_0\big((t-\xi_0)\varphi_0\big),\quad \varphi_2(t)=2\mathcal R_0\big((t-\xi_0)\varphi_1-\langle (t-\xi_0)\varphi_1,\varphi_0\rangle \varphi_0\big) \] 
 and $\mathcal R_0$ is  the regularized resolvent introduced in \eqref{eq:Reg-res}. Notice that $\varphi_0,\varphi_1$ and $\varphi_2$ are Schwartz functions (i.e.  in $\mathcal S(\R_+)$,  see  \cite[Appendix~A]{FH06}). 
The definition of $w_h$ involves the differential operator
\begin{equation}\label{eq:ub-L1,L2}
L_1^0(r,D_r)=\sin\theta D_r-\frac{1}2\cos\theta \gamma(r^2-\rho) 
\end{equation} 
and a function $\psi\in\mathcal  S(\R)$ defined via the ground state  $\psi_0$ of  the Montgomery model  in \eqref{eq:Mont} and the following phase function
\[\varphi(r)=\gamma\alpha(\theta)\Big(\frac{r^3}{6}+\frac{\rho r}{2} \Big)\,, \]
where
\[\alpha(\theta)=\frac{\sin\theta\cos\theta(1-\delta_0) }{\delta_0\sin^2\theta+\cos^2\theta }\,, \]
and $\delta_0$ the constant introduced in \eqref{eq:Th0}.  We define now the function $\psi(r)$ as follows
\[\psi(r)=\Big(\frac{c}{d} \Big)^{-1/12}\exp\big( i\varphi(r)\big)\,\psi_0\Big( \Big(\frac{c}{d} \Big)^{-1/6}r\Big) \]
where 
\begin{equation}\label{eq:def-psi}c=\cos^2\theta+\delta_0\sin^2\theta,\quad  d=\frac{\delta_0^2\gamma^2}{\delta_0\sin^2\theta+\cos^2\theta }\,, \end{equation}
and we choose (see \eqref{eq:Mont})
\begin{equation}\label{eq:ub-def-rho}
\rho= \Big(\frac{c}{d} \Big)^{1/3}\rho_0 .
\end{equation}
We conclude by  mentioning some estimates which follow easily from the definitions of $v$  and  $v_0$ in \eqref{eq:def-v} and \eqref{eq:ub-def-v0}:
\begin{equation}\label{eq:ub-est-v}
\begin{aligned}
\|v\|_{L^2(\R\times\R_+)}^2&=1+\mathcal O(h^{1/6})\,,\\
\int_{\R\times\R_+}r^kt^{n}|v|^2drdt&=\mathcal O(h^{\frac{k}3+\frac{n}2})\quad(k,n\geq 0)\,,\\
\int_{\R\times\R_+}|hD_rv|^2drdt&=\mathcal O(h^{5/3}), \quad\int_{\R\times\R_+}|hD_tv|^2=\mathcal O(h)\,.
\end{aligned}
\end{equation}
 
\subsubsection{Heuristics on  the construction of  $w_h$.}\label{sec:up-def-wh}
Starting from the definition of the  function  $v$ in \eqref{eq:ub-def-v0}, the quadratic form in \eqref{eq:HM-(12.8)} becomes (after neglecting error terms in the  magnetic potential)
\[ q^h_{M^{00}}(v)\approx h\tilde q^h(w_h)\]
where
\[\tilde q^h(w_h):=\int_{\R^2_+}\Big( |D_tw_h|^2+\big|\big(t-\xi_0-h^{1/6}L_1^0(r,D_r)\big)w_h\big|^2+h^{1/3}|L_2^0(h,D_r)w_h|^2 \Big)drdt\,,\]
$L_1^0(r,D_r)$ is introduced  in \eqref{eq:ub-L1,L2} and 
\[L_2^0=\cos\theta D_r+\frac12\sin\theta \gamma(r^2-\rho)\,.  \]
The construction of $w_h$  is based on minimizing
\[ \int_\R\left( \int_{\R_+} \Big(|D_tw_h|^2+\big|\big(t-\xi_0-h^{1/6}L_1^0(r,D_r)\big)w_h\big|^2\Big) dt\right)dr\,,\]
which amounts to finding the lowest eigenvalue of the operator
\[\mathcal  T_h:= D_t^2+\big(t-\xi_0-h^{1/6}L_1^0(r,D_r)\big)^2\,.\]
Writing
\[ \mathcal  T_h=D_t^2+(t-\xi_0)^2-2h^{1/6}(t-\xi_0)^2L_1^0(r,D_r)+h^{1/3}\big(L_1^0(r,D_r)\big)^2\,,\]
it is natural to search for $w_h$   in the  form in \eqref{eq:ub-def-wh} and satisfying 
\[ \mathcal T_hw_h- \left(\mu_0+\mu_1h^{1/6}L_1^0(r,D_r)+\mu_2^{1/3}\big(L_1^0(r,D_r)\big)^2\right)w_h\approx 0\]
in  the  following sense  (after taking the coefficients of $h^{i/6}$ to be $0$, for $i=0,1,2$)
\[ \begin{aligned}
 \big(D_t^2+(t-\xi_0)^2-\mu_0\big)\varphi_0&=0\\
 \big(D_t^2+(t-\xi_0)^2-\mu_0\big)\varphi_1&=\mu_1\varphi_0\\
 \big(D_t^2+(t-\xi_0)^2-\mu_0\big)\varphi_2&=\mu_2\varphi_0+\mu_1\varphi_1
\end{aligned}\]
Eventually, this leads to  $\mu_0=\Theta_0$, $\mu_1=0$, $\mu_2=\frac12\mu''(\xi_0)$  and $\varphi_0,\varphi_1,\varphi_2$ as in \eqref{eq:ub-def-wh}.

\subsection{Energy estimates}\label{sec:en-qm}

We will estimate the following energy arising  from Lemma~\ref{lem:10.2}:
\[
q^h_{\Ab^{00}}(\tilde u)=\int_{\R^2\times\R_+}\Big(|(hD_r-A_1^{00})\tilde u|^2+(1+2\kappa r)|(hD_s-A_2^{00})\tilde u|^2 +|hD_t\tilde u|^2\Big)drdsdt\,,\]
where $A_1^{00},A_2^{00}$ are introduced in \eqref{eq:up-A}. 

 Actually,  $q^h_{\Ab^{00}}(\tilde u)$ is bounded from above  by $q_{M^{00}}(v)$ modulo  error terms, where  $q_{M^{00}}(v)$  and  $v$  are  introduced in \eqref{eq:up-def-qf-M} and  \eqref{eq:def-tilde-u} respectively.  Due to the non-homogeneity of the magnetic field, the error terms involve a quantity\footnote{ This is $A(v)+B(v)$  appearing in \eqref{eq:ub-red-zeta=0}, which would equals  $0$ if the magnetic field were constant.} introduced in \eqref{eq:ub-red-zeta=0} whose control has to be done carefully.

Due  to  the phase terms in the definition of $\tilde u$ in \eqref{eq:def-tilde-u},  we have
\[ q^h_{\Ab^{00}}(\tilde u)
=\int_{\R^2\times\R_+}\Big(|hD_t\tilde u|^2+|(hD_s-A_{2,\rm  new}^{00})\tilde u|^2+|(hD_r-A_{1,\rm new}^{00})\tilde u|^2 \Big)drdsdt\]
where
\[ \left(\begin{array}{c}
A_{1,\rm new}^{00}\medskip\\
A_{2,\rm new}^{00}
\end{array}\right)=\left(\begin{array}{c}
M_{1,\zeta}^{00}\medskip\\
M_{2,\zeta}^{00}
\end{array}\right)+
\left(\begin{array}{c}
\check\kappa  \cos\theta st\medskip\\
\check\kappa  \sin\theta st
\end{array}\right)\,,  \]
and
\begin{equation}\label{eq:ub-M12}
\begin{aligned}
M_{1,\zeta}^{00}(r,t) &=\sin\theta( t-h^{1/2}\xi_0)+\zeta \cos\theta rt\\
M_{2,\zeta}^{00} (r,t) &=(1+2\kappa r)^{1/2}\Big(-\cos\theta (t-h^{1/2}\xi_0)\\&\qquad+ (\kappa\cos\theta+\zeta\sin\theta)  rt-\frac{\gamma}{2}(r^2-h^{2/3}\rho)\Big)\,.
\end{aligned}
\end{equation}
Since the function $s\mapsto \chi_h(s)$ is even, we have
\[\langle (hD_s-\check\kappa \sin\theta st) u,  M_{2,\zeta}^{00}u\rangle_{L^2(\R^2\times\R_+)}=0 \]
and
\[ \langle \check\kappa \cos\theta st \,u,  (hD_r-M_{1,\zeta}^{00})u\rangle_{L^2(\R^2\times\R_+)}=0\,.\]
Moreover,  we  have the estimates
\begin{align*}
\|(hD_s-\check\kappa \sin\theta st) \tilde u\|_{L^2(\R^2\times\R_+)}&\leq C \int_{\R\times\R_+}(h^{2-2\delta}+h^{2\delta} t^2)|v|^2drdt\\
&=\mathcal O(h^{2-2\delta}+h^{2\delta+1})\,, 
\end{align*}
and
\[ \|\check\kappa \cos\theta st\, u\|_{L^2(\R^2\times\R_+)}\leq Ch^{2\delta}  \int_{\R\times\R_+}t^2|v|^2drdt=\mathcal O(h^{2\delta+1}).\]
Notice that we used \eqref{eq:ub-est-v} and also that $|s|\leq C_0h^{\delta}$ in the support of $\tilde u$.   Consequently,  we get
\begin{equation}\label{eq:ub-step1} 
\begin{aligned}
q^h_{\Ab^{00}}(\tilde u)&\leq \int_{\R\times\R_+}\Big(|(hD_r-M_{1,\zeta}^{00})v|^2+ |M_{2,\zeta}^{00} v|^2+|hD_tv|^2\Big)drdt\\
&\qquad+\mathcal O\big(h^{2-2\delta}+h^{2\delta+1}\big)\,.
\end{aligned}
 \end{equation}
 Let us now reduce the computations to the potentials $M_1^{00}$ and $M_2^{00}$  in \eqref{eq:ub-def-M} which amount to $M_{1,\zeta}^{00}$ and $M_{2,\zeta}^{00}$ with $\zeta=0$. 
  A straightforward computation yields,
\begin{multline}\label{eq:ub-red-zeta=0}
\|(hD_r-M_{1,\zeta}^{00})v\|^2_{L^2(\R\times\R_+)}+\|M_{2,\zeta}^{00}v\|^2_{L^2(\R\times\R_+)}\\
=\|(hD_r-M_{1}^{00})v\|^2_{L^2(\R\times\R_+)}+\|M_{2}^{00}v\|^2_{L^2(\R\times\R_+)}+\zeta \big(A(v)+B(v)\big),
\end{multline}
where
\[
\begin{aligned}
A(v)&:=\zeta\cos^2\theta\,\| rt v\|^2_{L^2(\R\times\R_+)}-2\cos\theta\,{\rm Re}\langle (hD_r-M_1^{00})v\,,\,rtv \rangle_{L^2(\R\times\R_+)}\\
B(v)&:=\zeta\sin^2\theta\,\| rt v\|^2_{L^2(\R\times\R_+)}+2\sin\theta\,{\rm Re}\langle M_2^{00}v\,,\,rtv \rangle_{L^2(\R\times\R_+)}
\end{aligned}
\]
 and by \eqref{eq:ub-est-v}
 \[\| rt v\|^2_{L^2(\R\times\R_+)}=\mathcal O(h^{5/3}),\quad \langle hD_rv\,,\,rtv \rangle_{L^2(\R\times\R_+)}=\mathcal O( h^{5/3})\,.  \]
 So, we end up with estimating
 \[
 F(v):=\langle (\cos\theta M_1^{00}+\sin\theta M_2^{00})v\,,\,rtv \rangle_{L^2(\R\times\R_+)} \,.\]
 Notice that
\begin{multline*}
\cos\theta M_1^{00}(r,t)+\sin\theta M_2^{00}(r,t) =\cos\theta\sin\theta\big(1-(1+2\kappa r)^2 \big) (t-h^{1/2}\xi_0)\\
+(1+2\kappa r)^2\cos\theta\Big(\kappa\cos \theta rt-\frac{\gamma}{2}\Big(r^2-h^{2/3}\rho\Big)\Big)\,. \end{multline*}

By expanding 
\[ (1+2\kappa r)^{1/2}=1+\kappa r+\mathcal O(r^2)\quad (r\to0)\,,\]
 we observe  that,   for $|r|\leq r_0$ and $r_0$ sufficiently  small,
 \[ |\cos\theta M_1^{00}(r,t)+\sin\theta M_2^{00}(r,t)|\leq C(r^2+t^2+h^{2/3})\,,\]
 so we get by \eqref{eq:ub-est-v} and the Cauchy-Schwarz inequality that
 \[ F(v)=\mathcal O(h^{3/2})\,.\]
 Therefore,  $A(v)+B(v)=\mathcal O(h^{3/2})$ and we deduce from \eqref{eq:ub-red-zeta=0} and 
 \eqref{eq:ub-step1} that
 \begin{equation}\label{eq:ub-eff-qf}
q^h_{\Ab^{00}}(\tilde u)\leq q^h_{M^{00}}(v)+\mathcal O\big(h^{2-2\delta}+h^{2\delta+1}+h^{3/2}\big)\,,
\end{equation} 
where $q^h_{M^{00}}(v)$ is introduced in \eqref{eq:up-def-qf-M}.

\subsection{Conclusion}\label{sec:ub-conc}

Collecting \eqref{eq:ub-eff-qf} and \eqref{eq:HM-(12.8)},  we get 
\[q^h_{\Ab^{00}}(\tilde u)\leq (\Theta_0 h + \widehat{\gamma}_{0,\Bb}  h^{\frac{4}{3}} + Ch^{\frac{4}{3} + \eta}\Big)\|v\|_{L^2(\R\times\R_+)}^2+ R_h(v) \]
where
\[ R_h(v)=\mathcal O\big(h^{2-2\delta}+h^{2\delta+1}+h^{3/2}\big)=\mathcal O(h^{\frac43+\hat\eta})\]
for some $\hat\eta>0$, thanks to the condition $\frac5{18}<\delta<\frac13$.

We insert this into Lemma~\ref{lem:10.2} with $u$ given in \eqref{eq:up-def-qm-u}. Notice that $u$ satisfies \eqref{eq:cond-u} with $M_h(u)=\mathcal  O(1)$. So by Lemma~\ref{lem:10.2} and \eqref{eq:ub-norm-tilde-u}, we get for some $\eta_*>0$
\[ q^h_{\Ab}(u)\leq  (\Theta_0 h + \widehat{\gamma}_{0,\Bb}  h^{\frac{4}{3}} + Ch^{\frac{4}{3} + \eta_*}\Big)\|v\|_{L^2(\R\times\R_+)}^2\,. \] 
 Comparing \eqref{eq:ub-norm-tilde-u} and \eqref{eq:up-def-qm-u}, we get by \eqref{eq:norm-g},
\begin{equation}\label{eq:up-def-qm-u-norm}
\|u\|^2=\big(1+\mathcal O(h^{2\delta})\big)\|v\|^2_{L^2(\R\times\R_+)}\,.
\end{equation}

  Applying the min-max principle, and noticing that $1+2\delta>\frac43$ for $\frac5{18}<\delta<\frac13$, we finish the proof of \eqref{eq:gse-ub-3D}.

\subsection*{Acknowledgments}
Preliminary discussions of the first author  on this problem with Xingbin Pan more than twelve years ago are acknowledged. This work was initiated while the second author  visited the \emph{Laboratoire de Math\'ematiques Jean Leray} (LMJL)  at  \emph{Nantes Universit\'e}  in 2021.  The authors would like to thank the  support from the  \emph{F\'ed\'eration de  recherche  Math\'ematiques des Pays de Loire} and \emph{Nantes Universit\'e}.

\end{document}